\documentclass[12pt]{amsart}

\usepackage{amssymb, amsfonts, amsmath, array}
\usepackage{longtable}
\usepackage{url}
\usepackage[all,arc]{xy}
\usepackage{amscd}
\usepackage{mathrsfs}
\usepackage{geometry}
\usepackage[colorlinks,citecolor=blue]{hyperref}
\usepackage{comment}
\usepackage{enumerate}
\usepackage{graphicx}
\usepackage{bm}
\usepackage{color}
\usepackage{anyfontsize}
\usepackage{caption}

\numberwithin{equation}{section}

\theoremstyle{plain}
\newtheorem{theorem}{Theorem}[section]
\newtheorem{corollary}[theorem]{Corollary}
\newtheorem{lemma}[theorem]{Lemma}
\newtheorem{proposition}[theorem]{Proposition}

\newtheorem{definition}[theorem]{Definition}

\theoremstyle{remark}
\newtheorem{remark}{Remark}[section]

\begin{document}


\title[On the dimension of the space of static potentials]{On the dimension of the space of static potentials on three-manifolds}

\author{Vladimir Medvedev}

\address{Faculty of Mathematics, National Research University Higher School of Economics, 6 Usacheva Street, Moscow, 119048, Russian Federation}

\email{vomedvedev@hse.ru}



\begin{abstract}
We investigate the interplay between the dimension of the space of static potentials and the geometric and topological structure of the underlying static three-manifold. A partial classification of boundaryless static manifolds is obtained in terms of this dimension. We also treat the case of static manifolds with boundary. In particular, we prove that if a compact static manifold with boundary admits a static potential whose zero set is disjoint from the boundary, then the space of static potentials is necessarily one-dimensional. These results rely on a careful analysis of the relative positions of the zero sets of linearly independent static potentials -- a technique originally introduced by Miao and Tam.
\end{abstract}

\maketitle


\newcommand\cont{\operatorname{cont}}
\newcommand\diff{\operatorname{diff}}

\newcommand{\dvol}{\text{dA}}
\newcommand{\Ric}{\operatorname{Ric}}
\newcommand{\Hess}{\operatorname{Hess}}
\newcommand{\GL}{\operatorname{GL}}
\newcommand{\myO}{\operatorname{O}}
\newcommand{\myP}{\operatorname{P}}
\newcommand{\eye}{\operatorname{Id}}
\newcommand{\myF}{\operatorname{F}}
\newcommand{\Vol}{\operatorname{Vol}}
\newcommand{\odd}{\operatorname{odd}}
\newcommand{\even}{\operatorname{even}}
\newcommand{\ol}{\overline}
\newcommand{\mye}{\operatorname{E}}
\newcommand{\myo}{\operatorname{o}}
\newcommand{\myt}{\operatorname{t}}
\newcommand{\irr}{\operatorname{Irr}}
\newcommand{\mydiv}{\operatorname{div}}
\newcommand{\curl}{\operatorname{curl}}
\newcommand{\re}{\operatorname{Re}}
\newcommand{\im}{\operatorname{Im}}
\newcommand{\can}{\operatorname{can}}
\newcommand{\scal}{\operatorname{scal}}
\newcommand{\tr}{\operatorname{trace}}
\newcommand{\sgn}{\operatorname{sgn}}
\newcommand{\SL}{\operatorname{SL}}
\newcommand{\myspan}{\operatorname{span}}
\newcommand{\mydet}{\operatorname{det}}
\newcommand{\SO}{\operatorname{SO}}
\newcommand{\SU}{\operatorname{SU}}
\newcommand{\specl}{\operatorname{spec_{\mathcal{L}}}}
\newcommand{\fix}{\operatorname{Fix}}
\newcommand{\id}{\operatorname{id}}
\newcommand{\grad}{\operatorname{grad}}
\newcommand{\singsup}{\operatorname{singsupp}}
\newcommand{\wave}{\operatorname{wave}}
\newcommand{\ind}{\operatorname{ind}}
\newcommand{\mynull}{\operatorname{null}}
\newcommand{\inj}{\operatorname{inj}}
\newcommand{\arcsinh}{\operatorname{arcsinh}}
\newcommand{\Spec}{\operatorname{Spec}}
\newcommand{\Ind}{\operatorname{Ind}}
\newcommand{\Nul}{\operatorname{Nul}}
\newcommand{\inrad}{\operatorname{inrad}}
\newcommand{\mult}{\operatorname{mult}}
\newcommand{\Length}{\operatorname{Length}}
\newcommand{\Area}{\operatorname{Area}}
\newcommand{\Ker}{\operatorname{Ker}}
\newcommand{\floor}[1]{\left \lfloor #1  \right \rfloor}

\newcommand\restr[2]{{
  \left.\kern-\nulldelimiterspace 
  #1 
  \vphantom{\big|} 
  \right|_{#2} 
  }}
 
\section{Introduction} 
  
Let $(M^n, g)$ be a \textit{static manifold}, that is, a complete Riemannian manifold without boundary equipped with a nontrivial static potential $V$ satisfying
\begin{equation}\label{static}
\Hess_g V = \Delta_g V \, g + V \, \mathrm{Ric}_g,
\end{equation}
where $\Hess_g$ is the Hessian, $\Delta_g=\tr_g\Hess_g$ is the Laplacian of $g$, and $\mathrm{Ric}_g$ is its Ricci curvature. Static manifolds have garnered significant attention since their introduction in~\cite{hawking2023large} in the context of \textit{static spacetimes}. They also play a pivotal role in the study of scalar curvature deformations; see, e.g., \cite{fischer1975deformations,bourguignon1975stratification,corvino2000scalar,chrusciel2003mapping,corvino2006asymptotics,delay2011localized,brendle2011deformations,qing2016scalar}. 

The classification of static manifolds is a fundamental problem of both physical and mathematical significance. Although a complete classification has not yet been achieved, numerous significant results have been established \cite{kobayashi1981conformally,kobayashi1982differential,lafontaine1983geometrie,boucher1984uniqueness,shen1997note,lafontaine2009remark,qing2013note}. In this paper, we contribute to this classification program. 

Before stating our main results, we recall some fundamental properties of static manifolds.

The first geometric obstruction for a metric $g$ to be static is the scalar curvature $R_g$: it is known that $R_g$ must be constant; see, e.g.,~\cite{fischer1975deformations,tod2000spatial,corvino2000scalar}.

Taking the metric trace in~\eqref{static} and using the fact that $R_g$ is constant, we obtain
$$
\Delta_g V = -\frac{R_g}{n-1} V,
$$
so $V$ is a Laplace eigenfunction with eigenvalue $-\dfrac{R_g}{n-1}$. 
This yields a second necessary condition: the number $-\dfrac{R_g}{n-1}$ must belong to the spectrum of the Laplace operator on $(M,g)$.

A third obstruction concerns the structure of the set of all static potentials. This set is clearly a vector space, which we denote by $\mathcal{P}$. Using the above identity, equation~\eqref{static} can be rewritten as
$$
\Hess_g V = \left( \mathrm{Ric}_g - \frac{R_g}{n-1} g \right) V,
$$
so that
$$
\mathcal{P} = \left\{ f \in C^\infty(M) \mid \Hess_g f = \left( \mathrm{Ric}_g - \frac{R_g}{n-1} g \right) f \right\}.
$$
It is known that $\dim \mathcal{P} \leqslant n+1$ (see~\cite[Corollary 2.4]{corvino2000scalar}).

In this paper, we investigate the relationship between $\dim \mathcal{P}$ and the geometric and topological properties of $(M,g)$ in the three-dimensional case, i.e. when $n=3$. This question was previously studied is such papers as~\cite{fischer1975deformations,tod2000spatial,corvino2000scalar,lafontaine2009remark,miao2015static,wang2024static}. Notably,~\cite{tod2000spatial} derived a \textit{local} expression for a static metric with $\dim \mathcal P \geqslant 2$. Explicitly, in suitable local coordinates $(x, r, \phi)$, the metric takes the form
\[
    \left(k_1+\frac{k_2}{r}+k_3r^2\right)dx^2+\frac{dr^2}{k_1+\dfrac{k_2}{r}+k_3r^2}+r^2d\phi^2,
\]
where $k_1, k_2, k_3$ are real constants with clear geometric interpretations (see Section~\ref{meaning}). While this result is local, our primary interest lies in global results.

In~\cite{he2015uniqueness}, the authors investigated the \textit{space of virtual $(\lambda, n+m)$-Einstein warping functions}, where $n$ denotes the manifold dimension, $\lambda\in \mathbb{R}$, and $m\in\mathbb{R}_+$. This space generalizes the space of static potentials, $\mathcal P$. Specifically, on a Riemannian manifold $(M,g)$ with constant scalar curvature $R_g$, $\mathcal P$ is a subspace of the virtual $\left(\dfrac{R_g}{n-1}, n+1\right)$-Einstein warping functions satisfying $\Delta_gV=-\dfrac{R_g}{n-1}V$. Corollary 2.4 of~\cite{he2015uniqueness} implies that if $\dim\mathcal P=n+1$ and $(M^n,g)$ is complete, then $(M^n,g)$ is either a simply connected space form or a circle. In this paper, we recover this result for the specific case $n = 3$.

The first main result of this work is the following theorem.

\begin{theorem}\label{T1}
Assume $\dim \mathcal{P}\geqslant 2$, and let $f_1, f_2$ be two linearly independent static potentials on $(M,g)$. If the zero sets intersect, i.e., $f_1^{-1}(0) \cap f_2^{-1}(0) \ne \varnothing$, then $(M,g)$ is a space form. Specifically, if $R_g > 0$, then $(M,g)$ is homothetic to the standard sphere and $\dim \mathcal{P}=4$. If $R_g < 0$, then $(M,g)$ is either homothetic to hyperbolic space with $\dim \mathcal{P}=4$, or a quotient thereof by a cyclic group generated by a hyperbolic or parabolic element in $O(2,1)$, in which case $\dim \mathcal{P}=2$. 
\end{theorem}

We note in passing that the quotient manifolds arising when $R_g<0$ have the following topological and geometric properties (see, e.g., Chapters 3-5 in~\cite{thurston2022geometry}). The quotient of $\mathbb H^3$ by a cyclic group generated by a hyperbolic element in $O(2,1)$ is a complete hyperbolic 3-manifold homeomorphic to $\mathbb S^1\times \mathbb R^2$. It contains a single closed geodesic and has infinite volume. The quotient of $\mathbb H^3$ by a cyclic group generated by a parabolic element in $O(2,1)$  is also homeomorphic to $\mathbb S^1\times \mathbb R^2$ but geometrically it has a cusp (a rank-1 cusp) instead of a closed geodesic. The metric near the cusp is modelled on a horosphere quotient.

The case where $R_g<0$ was recently considered in~\cite[Theorem 4.1]{wang2024static}. The case where $R_g=0$ was previously considered in~\cite{miao2015static}: in the assumption on $f_1$ and $f_2$ as in the previous theorem, the manifold must be flat. We complete this result by giving a precise description of all 3-dimensional flat static manifolds satisfying the above assumptions.

\begin{theorem}\label{T2}
Let $R_g=0$. Assume $\dim \mathcal{P}\geqslant 2$, and let $f_1, f_2$ be two linearly independent static potentials. If the zero sets intersect, i.e., $f_1^{-1}(0) \cap f_2^{-1}(0) \ne \varnothing$, then $(M,g)$ is isometric to either $\mathbb E^3$ and $\dim\mathcal{P}=4$ or $\mathbb{S}^1 \times \mathbb{E}^2$ (endowed with the product metric) and $\dim\mathcal{P}=3$.
\end{theorem}

We prove this theorem in Appendix~\ref{appC}.

For convenience in subsequent discussions, we introduce the notion of \textit{nice} manifolds.

\begin{definition}\label{def1}
We say that a Riemannin three-manifold $(M,g)$ is \emph{nice} if the eigenvalues of its Ricci tensor are of the form $(\lambda,\lambda,0)$, where $\lambda=const$. 
\end{definition}

An example of a nice metric is a flat metric, for which \(\lambda = 0\); this case was previously considered in~\cite{miao2015static}. Other examples include product metrics on \(\mathbb{R} \times \Sigma\), where \(\Sigma\) is either a round sphere or a hyperbolic plane of some radius, as well as quotients of such products. Moreover, by the Cheeger-Gromol splitting theorem, when \(\lambda > 0\), the universal cover of such a manifold splits isometrically as $\mathbb{R} \times \mathbb S^2$.

The second main result of this work is the following theorem.

\begin{theorem}\label{T3}
Assume $\dim \mathcal{P} \geqslant 2$, and let $f_1, f_2$ be two linearly independent static potentials on $(M,g)$. 
If $f_1^{-1}(0)\neq\varnothing$ is closed and $f_1^{-1}(0) \cap f_2^{-1}(0) = \varnothing$, then $(M,g)$ is nice and $\dim \mathcal{P} = 2$. 
Specifically, $(M,g)$ is homothetic to one of the following product manifolds (endowed with the product metric):
\begin{itemize}
    \item $R_g > 0$: $\mathbb{R} \times \mathbb{S}^2$, $\mathbb{S}^1 \times \mathbb{S}^2$, or $\mathbb{S}^1 \times \mathbb{R}P^2$.
    \item $R_g = 0$: $\mathbb{R} \times \mathbb{T}^2$, where $\mathbb{T}^2$ is a flat torus, or $\mathbb{R} \times \mathbb{KL}$, where $\mathbb{KL}$ is a flat Klein bottle.
\end{itemize}
\end{theorem}

\begin{remark}
The assumptions of the previous theorem also include the case where $f_2^{-1}(0)=\varnothing$. 
\end{remark}

The classification in the case where $R_g<0$ is a delicate question: to the best of our knowledge, nice Riemannian manifolds with $\lambda<0$ are not fully classified (see e.g~\cite{brooks20253} for some partial results). Examples of nice static manifolds are $\mathbb{R} \times \mathbb{H}^2$ and $\mathbb{R} \times \Sigma$, where $\Sigma$ is a hyperbolic surface, endowed with the product metric

Observe that theorems~\ref{T1} and~\ref{T3} imply Theorem 1.1 from \cite{lafontaine2009remark} for the case where $R_g>0$ and $(M,g)$ is closed. 

Notice that the assumptions on the zero sets of static potentials in Theorems~\ref{T1}, \ref{T2} and \ref{T3} are mutually excluded unless $(M,g)$ is flat. 

\begin{corollary}
Assume $\dim \mathcal{P}\geqslant 2$. Let $f_1, f_2$ and $f_3,f_4$ be two pairs of linearly independent static potentials $(M,g)$. Let $f_1^{-1}(0) \cap f_2^{-1}(0) \ne \varnothing$, $f_3^{-1}(0)\neq \varnothing$ is closed and $f_3^{-1}(0) \cap f_4^{-1}(0)=\varnothing$. Then $(M,g)$ is flat.
\end{corollary}

Finally, we notice that the assumption that $f_1^{-1}(0)$ is closed in Theorem~\ref{T3} is essential. In Section~\ref{examples} we provide some examples of static manifolds with $\dim\mathcal P=2$ such that the zero sets of two linearly independent static potentials are disjoint but both not compact. For these manifolds one has $R_{11}=R_{22}\neq R_{33}$ but $R_{ii}\neq const,\ i=1,2,3$, i.e. $(M,g)$ is not nice.

\medskip

In the second part of the paper, we focus on \textit{static manifolds with boundary}. 
Namely, we study Riemannian metrics $g$ on $M$ for which the following overdetermined Robin problem admits nontrivial solutions:
\begin{equation}\label{b-static}
\left\{
\begin{array}{rcl}
\Hess_g V &=& \Delta_g V \, g + V \, \mathrm{Ric}_g \quad \text{in } M, \\
\dfrac{\partial V}{\partial \nu} \, g_{\partial M} &=& B_{\partial M} \, V \quad \text{on } \partial M.
\end{array}
\right.
\end{equation}
Here, $g_{\partial M}$ and $B_{\partial M}$ denote the induced metric and the second fundamental form of $\partial M$ with respect to the outward unit normal $\nu$, and $H_{\partial M} = \tr_{g_{\partial M}} B_{\partial M}$ is the corresponding mean curvature. Our sign convention is chosen so that the unit sphere in Euclidean space has mean curvature equal to $2$.

As in the boundaryless case, we refer to solutions of~\eqref{b-static} as \textit{static potentials}. 
Slightly abusing notation, we denote by $\mathcal{P}$ the space of static potentials on $(M,g)$.

Taking the metric trace, we find that any static potential $V$ satisfies the \textit{Robin eigenvalue problem}
\begin{equation}\label{Robin}
\left\{
\begin{array}{rcl}
\Delta_g V &=& -\dfrac{R_g}{n-1} \, V \quad \text{in } M, \\
\dfrac{\partial V}{\partial \nu} &=& \dfrac{H_{\partial M}}{n-1} \, V \quad \text{on } \partial M.
\end{array}
\right.
\end{equation}
It is known that for a static manifold with boundary $R_g$ and $H_{\partial M}$ are some constants (see, e.g., Proposition 1 in~\cite{cruz2023static}).

The theory of static manifolds with boundary develops in close analogy with the classical theory of boundaryless static manifolds. In particular, a number of recent works -- such as~\cite{cruz2019prescribing,ho2020deformation,huang2022scalar,cruz2023critical,sheng2024localized} -- focus on the deformation theory of the curvature functional
$$
g \in \mathcal{R}(M) \longmapsto (R_g,\, H_{\partial M}),
$$
where $\mathcal{R}(M)$ denotes the space of Riemannian metrics on a manifold $M$ with boundary. Meanwhile, the geometric structure and rigidity properties of static manifolds with boundary are investigated in~\cite{cruz2023static,medvedev2024static,sheng2025static,medvedev2025some,sheng2026obata}. 

The term ``static manifold with boundary'' was coined in~\cite{almaraz2025rigidity}, alongside a study of deformations of the above functional and applications to general relativity. Such manifolds are relevant to the study of \textit{photon surfaces} (see e.g. \cite{claudel2001geometry,cederbaum2015uniqueness,raulot2021spinorial,cederbaum2025uniqueness}) and arise naturally in the context of the \textit{Bartnik conjecture}, as shown in~\cite{huang2022scalar}.

Our contribution in this context is the following theorem.

\begin{theorem}\label{thmbstat}
Let $(M,g)$ be a compact static manifold with boundary. If the zero set of a static potential is not empty and disjoint from $\partial M$, then $\dim \mathcal{P} = 1$.
\end{theorem}

\subsection*{Conventions}
Throughout this paper, all manifolds are assumed to be smooth and connected. 
All metrics are smooth and complete Riemannian metrics.

\subsection*{Acknowledgments} The author is grateful to Lucas Ambrozio for suggesting this problem and for his careful guidance. Results of the project ``Symmetry. Information. Chaos", carried out within the framework of the Basic Research Program at HSE University in 2026, are presented in this work.

\section{Static manifolds without boundary}

In this section we extend the results of Proposition 2.1, Lemma 2.1-2.5, Proposition 2.2-2.3, and Theorem 2.1 from \cite{miao2015static} (see also~\cite{tod2000spatial}) to the general case where $R_g$ is not assumed to vanish. The necessary modifications to the original arguments are minor; in fact, the structure of all proofs remains essentially unchanged except at those points where the condition $R_g = 0$ was explicitly used. Some proofs for arbitrary $R_g$ were given in~\cite{wang2024static}. We present full proofs only for those parts where the vanishing of the scalar curvature played a role, and omit the rest as they carry over verbatim.

\begin{remark}
\item Very recently, the same technique was applied to the so-called $m$-quasi-Einstein manifolds. For more details see~\cite{gonccalves2025remarks}. 
\end{remark}

In what follows, let $(M,g)$ be a Riemannian manifold. We denote the components of the Ricci tensor by $R_{ij}$. Covariant differentiation is denoted by a semicolon; for instance, $R_{ij;k}$ and $Z_{;k}$ represent the covariant derivatives of the Ricci tensor and $Z$, respectively.

\begin{proposition}\label{prop1}
Let $\{e_1, e_2, e_3\}$ be an orthonormal frame diagonalizing the Ricci curvature at a point $p$.
\begin{enumerate}
    \item[(i)] If $f$ is a static potential, then
    $$
    f(R_{33;1} - R_{31;3}) = (R_{22} - R_{33})f_{;1},\quad
    f(R_{11;2} - R_{12;1}) = (R_{33} - R_{11})f_{;2},
    $$
    $$
    f(R_{22;3} - R_{23;2}) = (R_{11} - R_{22})f_{;3}.
    $$
    \item[(ii)] Assume that $\{R_{11}, R_{22}, R_{33}\}$ are distinct, and let $N, V$ be two non-vanishing static potentials in a domain. Then $V = cN$ for some constant $c$.
    \item[(iii)] Suppose $R_{11} = R_{22} \ne R_{33}$ and let $N$ be a positive static potential. If $f$ is another static potential, then $Z = N^{-1}f$ satisfies $Z_{;1} = Z_{;2} = 0$.
\end{enumerate}
\end{proposition} 

The proof can be found in Appendix in \cite{wang2024static}. Notice that the tensor whose components in an orthonormal frame are the expressions in the brackets in the lefthand side of (i) is exactly the \textit{Cotton-York tensor} (up to a multiplicative constant). 

\begin{lemma}\label{lem1}
Let $f$ be a static potential whose zero set is nonempty, and set $\Sigma = f^{-1}(0)$.
\begin{enumerate}
    \item[(i)] The surface $\Sigma$ is totally geodesic, and $|\nabla^g f|$ is a positive constant on every connected component of $\Sigma$.
    \item[(ii)] For any $p \in \Sigma$, the gradient $\nabla^g f(p)$ is an eigenvector of the Ricci tensor $\mathrm{Ric}_g$ at $p$.
    \item[(iii)] Fix $p \in \Sigma$, and let $\{e_1, e_2, e_3\}$ be an orthonormal frame at $p$ that diagonalizes $\mathrm{Ric}_g$, with $e_3$ orthogonal to $\Sigma$. Then $R_{11} = R_{22}$ at $p$.
    \item[(iv)] Denote by $K$ the Gauss curvature of $\Sigma$ at $p$. With the same frame as in (iii), one has
    $$
    K = 2R_{11} - \dfrac{R_g}{2} = 2R_{22} - \dfrac{R_g}{2} = -R_{33} + \dfrac{R_g}{2}.
    $$
    Consequently, $K = \dfrac{R_g}{2}$ holds precisely when $R_{11} = R_{22} = \dfrac{R_g}{2}$ and $R_{33} = 0$ at $p$; and $K = \dfrac{R_g}{6}$ occurs if and only if $R_{11} = R_{22} = R_{33} = \dfrac{R_g}{3}$ at $p$.
\end{enumerate}
\end{lemma}

\begin{proof}
$(iv)$ The identities follow immediately from the contracted Gauss equation, the decomposition $R_g = R_{11} + R_{22} + R_{33}$, and the equality $R_{11} = R_{22}$ established in $(iii)$.
\end{proof}

\begin{lemma}\label{lem2}
If the Ricci curvature tensor of $g$ possesses three distinct eigenvalues at a point of $M$, then $\dim \mathcal{P}\leqslant 1$. The space $\mathcal{P}$ is determined by the Ricci and Cotton-York tensors.
\end{lemma}

The last statement in the previous lemma follows from Corollary 3 in~\cite{tod2000spatial}.

\begin{lemma}\label{lem3}
Let $f$ and $N$ be static potentials, and assume that $N$ does not vanish in a domain $U$. Define $Z = f/N$. Then either $Z$ is constant in $U$, or $\nabla^g Z$ is nowhere zero in $U$. In the latter case:
\begin{enumerate}
    \item[(i)] Every level set of $Z$ contained in $U$ is a totally geodesic surface.
    \item[(ii)] The quantity $N^2 |\nabla^g Z|^2$ is constant on each connected component of any level set of $Z$ in $U$.
    \item[(iii)] The manifold $(M,g)$ is locally isometric in $U$ to a warped product $((-\varepsilon,\varepsilon) \times \Sigma,\, N^2 dt^2 + g_0)$, where $\Sigma$ is a two-dimensional surface, $g_0$ is a fixed metric on $\Sigma$, and $Z$ is constant on each slice $\Sigma_t = \{t\} \times \Sigma$.
\end{enumerate}
\end{lemma}

\begin{proof}
It suffices to establish formula (2.6) from \cite{miao2015static}.

Let $\{x^i\}$ be local coordinates on $M$. Since both $N$ and $f = NZ$ satisfy the static equation~\eqref{static}, we have
$$
-\frac{NZ}{2} g_{ij} + NZ R_{ij} = (NZ)_{;ij}
= -\frac{NZ}{2} g_{ij} + NZ R_{ij} + N Z_{;ij} + N_{;i} Z_{;j} + N_{;j} Z_{;i}.
$$
Canceling common terms yields
$$
N Z_{;ij} = -N_{;i} Z_{;j} - N_{;j} Z_{;i},
$$
which is equivalently expressed as
$$
N \, \mathrm{Hess}_g Z(X, Y) = -\langle \nabla^g N, X \rangle \langle \nabla^g Z, Y \rangle - \langle \nabla^g N, Y \rangle \langle \nabla^g Z, X \rangle
$$
for all tangent vector fields $X, Y$.
\end{proof}

\begin{proposition}\label{prop2}
If $(M,g)$ is not Einstein at some point, then $\dim \mathcal{P}\leqslant 2$.
\end{proposition}

\begin{proof}
Assume, for contradiction, that $\dim \mathcal{P}> 2$, and choose three linearly independent static potentials $f_1, f_2, f_3$.
Let $U \subset M$ be an open set where $g$ is not Einstein at any point.
By Lemma~\ref{lem1}, the intersection of the complement of the union of the zero sets $\bigcup_{i=1}^3 f_i^{-1}(0)$ and $U$ is not empty.
Thus, there exists a connected open subset $\Omega \subset U$ on which each $f_i$ is nowhere zero.

Denote by $\{\lambda_1, \lambda_2, \lambda_3\}$ the eigenvalues of $\mathrm{Ric}_g$ on $\Omega$.
They cannot all be distinct: otherwise, Proposition~\ref{prop1}(ii) would force any two nonvanishing static potentials on $\Omega$ to be scalar multiples of one another, contradicting the linear independence of $f_1, f_2, f_3$.
At the same time, since $g$ is not Einstein on $\Omega$, the eigenvalues cannot all coincide.

Hence, after possibly reordering indices, we have $\lambda_1 = \lambda_2 \ne \lambda_3$ throughout $\Omega$.
Set $Z_1 = f_1 / f_3$ and $Z_2 = f_2 / f_3$.
By Proposition~\ref{prop1}(iii), the gradients $\nabla^g Z_1$ and $\nabla^g Z_2$ are both everywhere colinear with the eigenvector of $\mathrm{Ric}_g$ associated with $\lambda_3$.
Therefore, at each $q \in \Omega$, the vectors $\nabla^g Z_1(q)$ and $\nabla^g Z_2(q)$ are linearly dependent; i.e., there exists a scalar $\alpha$ such that
$$
\nabla^g Z_1 + \alpha \nabla^g Z_2 = 0 \quad \text{at } q.
$$
Further, the function $f_1+\alpha f_2$ is also a static potential and the gradient of $(f_1+\alpha f_2)/f_3$ is exactly $\nabla^g Z_1 + \alpha \nabla^g Z_2$, which vanishes at $q\in \Omega$. By Lemma~\ref{lem3}, $Z_1 + \alpha Z_2=\beta=const$ in $\Omega$. Consequently, $f_1 + \alpha f_2 = \beta f_3$ for some constant $\beta$, which contradicts the assumed linear independence of $f_1, f_2, f_3$.

Thus, $\dim \mathcal{P}\leqslant 2$ whenever $(M,g)$ is not Einstein at some point.
\end{proof}

\begin{lemma}\label{lem4}
Let $f$ and $\tilde{f}$ be static potentials, and assume that the zero set of $\tilde{f}$ is nonempty. Set $\Sigma = \tilde{f}^{-1}(0)$. Then, along $\Sigma$,
$$
\mathrm{Hess}_\Sigma f = \frac{1}{2} \left(K - \frac{R_g}{2}\right) f \, g_\Sigma,
$$
where $\mathrm{Hess}_\Sigma$ denotes the Hessian with respect to the induced metric $g_\Sigma$ on $\Sigma$, and $K$ is the Gauss curvature of $(\Sigma, g_\Sigma)$. In particular, the quantity $\left(K - \dfrac{R_g}{6}\right) f^3$ is constant on each connected component of $\Sigma$.
\end{lemma}

For the proof we refer the reader to Appendix in~\cite{wang2024static}.

At this point, we also observe that the equation in Lemma~\ref{lem4} coincides with equation (40) in \cite{ambrozio2017static}, discussed in Appendix~B.

\begin{lemma}\label{lem5}
Suppose $(\Sigma_0, g_0)$ is a surface of constant Gauss curvature $K=\dfrac{R_g}{2}\neq 0$. If $\dim \mathcal{P}\geqslant 2$ on
$$
(M,g) = ((-\varepsilon,\varepsilon) \times \Sigma,\, N^2 dt^2 + g_0),
$$
where $N > 0$. If $(M,g)$ is nice along $\Sigma_0$, then it is nice for all $\varepsilon>0$ sufficiently small. 
\end{lemma}

\begin{proof}
Fix a point $(t, q) \in (-\varepsilon,\varepsilon) \times \Sigma$, and let $\Sigma_t = \{t\} \times \Sigma$ denote the corresponding slice. Observe that $\Sigma_t$ is totally geodesic and has constant Gauss curvature.

Choose an orthonormal frame $\{e_1, e_2, e_3\}$ at $(t,q)$ that diagonalizes $\mathrm{Ric}_g$, with $e_1, e_2 \in T\Sigma_t$ and $e_3$ normal to $\Sigma_t$. The contracted Gauss equation applied to $\Sigma_t$ gives $R_{33} = 0$. Since $R_g = R_{11} + R_{22} + R_{33}$, it follows that
$$
R_{11} + R_{22} = R_g.
$$

Assume, for contradiction, that $R_{11} \ne R_{22}$ and that neither vanishes. Then the Ricci tensor has three distinct eigenvalues at $(t,q)$, and Lemma~\ref{lem2} would imply $\dim \mathcal{P}\leqslant 1$, which contradicts the standing assumption $\dim \mathcal{P}\geqslant 2$. Hence, either $R_{11} = R_{22}$ or one of them is zero. Combining this with $R_{11} + R_{22} = R_g$, the only possibilities (up to relabelling) are:
$$
R_{11} = R_{22} = \frac{R_g}{2}, \quad \text{or} \quad R_{11} = 0,\; R_{22} = R_g.
$$
Since $R_{11}=R_{22}=\dfrac{R_g}{2}\neq 0$ along $\Sigma_0$, by assumption, the latter case is excluded for all $\varepsilon>0$ sufficiently small (by continuity of Ricci eigenvalues). Therefore, we must have $R_{11} = R_{22} = \dfrac{R_g}{2}$ and $R_{33} = 0$ at every point $(t,q)$, i.e., it is nice.
\end{proof}

\begin{lemma}\label{lem6}
Suppose $(\Sigma_0, g_0)$ is a surface of constant Gauss curvature $K=\dfrac{R_g}{6}$. If $\dim \mathcal{P}\geqslant 2$ on
\[
(M,g) = ((-\varepsilon,\varepsilon) \times \Sigma,\, N^2 dt^2 + g_0),
\]
where $N > 0$, then $(M,g)$ is a space form.
\end{lemma}

\begin{proof}
Adopting the same notation and reasoning as in the preceding lemma, we deduce that $R_{33} = \dfrac{R_g}{3}$. Since $R_g = R_{11} + R_{22} + R_{33}$, it follows that
$$
R_{11} + R_{22} = \frac{2R_g}{3}.
$$

Suppose that $R_{11} \ne R_{22}$ and that neither equals $\dfrac{R_g}{3}$. Then the Ricci tensor would have three distinct eigenvalues at $(t,q)$, and Lemma~\ref{lem2} would force $\dim \mathcal{P}\leqslant 1$, contradicting the assumption $\dim \mathcal{P}\geqslant 2$. Hence, either $R_{11} = R_{22}$ or one of them equals $\dfrac{R_g}{3}$.

In the first case, $R_{11} = R_{22}$ together with $R_{11} + R_{22} = \dfrac{2R_g}{3}$ gives $R_{11} = R_{22} = \dfrac{R_g}{3}$.  
In the second case, say $R_{11} = \dfrac{R_g}{3}$; then $R_{22} = \dfrac{2R_g}{3} - R_{11} = \dfrac{R_g}{3}$ as well. Thus, in either case,
$$
R_{11} = R_{22} = R_{33} = \frac{R_g}{3},
$$
which means that $(M,g)$ is Einstein at $(t,q)$. Since the point was arbitrary, $(M,g)$ is a space form.
\end{proof}

\begin{proposition}\label{prop3}
Assume $\dim \mathcal{P}\geqslant 2$, and let $f_1, f_2$ be two linearly independent static potentials. Denote by $P_1$ and $P_2$ connected components of $f_1^{-1}(0)$ and $f_2^{-1}(0)$, respectively. If $P_1 \cap P_2 \ne \varnothing$, then
\begin{enumerate}
    \item[(i)] $(M,g)$ is Einstein along $P_1 \cup P_2$;
    \item[(ii)] $(M,g)$ is Einstein in an open neighborhood of $P_1 \setminus f_2^{-1}(0)$ and of $P_2 \setminus f_1^{-1}(0)$.
\end{enumerate}
\end{proposition}

\begin{proof}
Since $f_1$ and $f_2$ are linearly independent, their gradients $\nabla^g f_1$ and $\nabla^g f_2$ are linearly independent at every point of $f_1^{-1}(0) \cap f_2^{-1}(0)$. Hence, the intersection $P_1 \cap P_2$ is a smooth embedded curve. As both $P_1$ and $P_2$ are totally geodesic (Lemma~\ref{lem1}(i)), this curve is a geodesic.

Let $K_1$ and $K_2$ denote the Gauss curvatures of $P_1$ and $P_2$, respectively. By Lemma~\ref{lem4}, the quantities $\left(K_1 - \dfrac{R_g}{6}\right) f_2^3$ and $\left(K_2 - \dfrac{R_g}{6}\right) f_1^3$ are constant on $P_1$ and $P_2$, respectively. Since $f_1 = f_2 = 0$ on $P_1 \cap P_2$, both constants must vanish. The zero sets $P_1 \cap f_2^{-1}(0)$ and $P_2 \cap f_1^{-1}(0)$ contain nontrivial curves, so the vanishing of these constants implies
$$
K_1 = \frac{R_g}{6} \quad \text{and} \quad K_2 = \frac{R_g}{6} \quad \text{on } P_1 \cup P_2.
$$
By Lemma~\ref{lem1}(iv), this yields $R_{11} = R_{22} = R_{33} = \dfrac{R_g}{3}$ along $P_1 \cup P_2$, i.e., $(M,g)$ is Einstein there. This establishes (i).

To prove (ii), let $p \in P_1 \setminus f_2^{-1}(0)$. Then $f_2$ is nonvanishing in a neighborhood $U$ of $p$. Define $Z = f_1 / f_2$ on $U$; note that $Z = 0$ on $P_1 \cap U$. By Lemma~\ref{lem3}(iii), there exists an open neighborhood $\Omega \subset U$ of $p$, diffeomorphic to $(-\varepsilon, \varepsilon) \times \Sigma$, where $\Sigma \subset P_1$ is a small surface patch containing $p$, such that $Z$ is constant on each slice $\{t\} \times \Sigma$, and the metric on $\Omega$ takes the warped product form
$$
g = f_2^2 \, dt^2 + g_0,
$$
with $g_0$ the induced metric on $\Sigma$. From (i), $(\Sigma, g_0)$ has constant Gauss curvature $R_g/6$.

Since $\dim \mathcal{P}\geqslant 2$ on $\Omega$, Lemma~\ref{lem6} implies that the metric is Einstein throughout $\Omega$. An identical argument applies to points in $P_2 \setminus f_1^{-1}(0)$. This proves (ii).
\end{proof}

We now ready to prove Theorem~\ref{T1}.

\begin{proof}
Set $S = f_1^{-1}(0) \cap f_2^{-1}(0)$. Since $f_1$ and $f_2$ are linearly independent, their gradients are linearly independent along $S$, so $S$ is a smooth embedded curve. Consequently, its complement $M \setminus S$ is path-connected.

For any point $p \in M \setminus S$, at least one of $f_1(p)$ or $f_2(p)$ is nonzero. Hence, in a neighborhood of $p$, one of the potentials is nonvanishing, and the metric $g$ is real-analytic there (as static metrics are analytic in harmonic coordinates see~\cite[Proposition 2.8]{corvino2000scalar}). By Proposition~\ref{prop3}(ii), $(M,g)$ is Einstein in a neighborhood of every point of $M \setminus S$. Since the condition of being Einstein is closed and $M \setminus S$ is dense and connected, it follows that $(M,g)$ is Einstein in on all of $M$, i.e., it is a space form.

The final classification follows from the standard rigidity of complete, simply connected $3$-dimensional space forms: if $R_g > 0$, the universal Riemannian cover is the unit sphere (up to homothety); if $R_g < 0$, it is hyperbolic space (up to homothety). 

Recall that $\mathcal P$ on $\mathbb S^3$ or $\mathbb H^3$ is generated by the coordinate functions $x_0,\ldots, x_3$ in 4-dimensional Euclidean or Minkowski space of signature $(-,+,+,+)$, respectively. Suppose there is a linear combination  \( L(x) = a_0x_0 + \cdots + a_3x_3 \in span\{x_0,\ldots, x_3\} \) which is preserved by a subgroup of the isometry group acting freely and properly discontinuously. $L$ corresponds to a fixed nonzero vector \( v = (a_0, \ldots, a_3) \) in the ambient space. A subgroup \( \Gamma \) of the isometry group preserves \( L \) if and only if it fixes \( v \) (i.e., \( \Gamma \) is contained in the stabiliser of \( v \)). 

Consider $\mathbb S^3$ in Euclidean space. Any nonzero vector \( v \) has stabiliser \( O(3) \) (rotations around the axis through \( v \)). This group fixes the two antipodal points \( \pm v / |v| \) on the sphere. Every nontrivial element of \( O(3) \) has fixed points (the axis), so no subgroup (except the trivial one) preserving $v$ can act freely on \( \mathbb S^3 \). Hence no static manifolds are obtained as quotient of $\mathbb S^3$ (except $\mathbb S^3$).

Consider $\mathbb H^3$ in Minkowski space. If \( v \) is spacelike (\( v \cdot v > 0 \)), its stabiliser is \( O(2,1) \). Inside it one can find infinite discrete subgroups acting freely and properly discontinuously on \( \mathbb{H}^3 \), e.g., a cyclic group generated by a hyperbolic element (a pure boost). Such an element has no fixed points in \( \mathbb{H}^3 \) and the quotient \( \mathbb{H}^3 / \Gamma \) is a smooth hyperbolic 3-manifold.  If \( v \) is lightlike (\( v \cdot v = 0 \)), its stabiliser is the Euclidean group \( \operatorname{Isom}(\mathbb{E}^2) \). It contains parabolic subgroups (e.g., translations along horospheres) that also act freely and properly discontinuously, giving smooth quotients. Timelike vectors (\( v \cdot v < 0 \)) give stabiliser \( O(3) \), which fixes a point in \( \mathbb{H}^3 \) and cannot act freely. Hence $v$ is lightlike or spacelike and we are interested in the subgroups in \( \operatorname{Isom}(\mathbb{E}^2)\) or \( O(2,1) \), respectively, acting freely and properly discontinuously. Consider these two cases separately.

\medskip

\textit{Case 1}: \(v\) is lightlike. The stabiliser of a lightlike vector is isomorphic to the Euclidean group \(\operatorname{Isom}(\mathbb{E}^2) = \mathbb{R}^2 \rtimes O(2)\). Its elements are: reflections, rotations, glide reflections, and translations. But rotations and reflections clearly have fixed points and glide reflections do not form a subgroup.

Hence \(\Gamma\) is contained in the translation subgroup \(\mathbb{R}^2\). The translations act on the orthogonal complement \(v^\perp/(\mathbb{R}v)\) (a Euclidean plane) as ordinary translations, so they have no nonzero fixed vectors in that plane. The only common fixed vector is \(v\) itself. Consequently, $\dim \mathcal P=\dim V^\Gamma = 1$. But this contradicts the assumption that $\dim \mathcal P \geqslant 2$. 

Hence $v$ cannot be lightlike and the final result is described by \textit{Case 2} below.

\medskip

\textit{Case 2}: \(v\) is spacelike. Write
\[
\mathbb{R}^{3,1} = \mathbb{R}v \oplus W,\qquad W = v^\perp,
\]
where \(W\) inherits a Lorentzian metric of signature \((-,+,+)\). Every \(\gamma\in\Gamma\subset O(2,1)\) acts trivially on \(v\) and as a Lorentz transformation on $W$. Then $\gamma$ can be:        
\begin{itemize}
    \item elliptic, i.e., it fixes a timelike line in \(W\); then it has fixed points in \(\mathbb{H}^3\) (a geodesic), hence cannot belong to a free action;
    \item hyperbolic, i.e., it fixes a spacelike line in \(W\) and has no timelike fixed vectors; there are no fixed points in \(\mathbb{H}^3\);
    \item parabolic, i.e., it fixes a lightlike line in \(W\) and has no timelike fixed vectors; there are no fixed points in \(\mathbb{H}^3\).
\end{itemize}
Thus every nontrivial \(\gamma\in\Gamma\) is either hyperbolic or parabolic. Each such element fixes a unique line \(\ell_\gamma\subset W\).

Let
\[
W^\Gamma = \{ w\in W \mid \gamma w = w\ \forall\gamma\in\Gamma \}.
\]
Then \(V^\Gamma = \mathbb{R}v \oplus W^\Gamma\). For any two elements \(\gamma_1,\gamma_2\in\Gamma\) we have \(\ell_{\gamma_1}\cap\ell_{\gamma_2}\) is either \(\{0\}\) (if the lines are distinct) or the line itself (if they coincide). Hence \(W^\Gamma\) is the intersection of all these lines over \(\gamma\in\Gamma\setminus\{\mathrm{id}\}\). Therefore:
\[
\dim W^\Gamma = 
\begin{cases}
0 &\text{if not all elements share a common fixed line},\\
1 &\text{if they do}.
\end{cases}
\]
Consequently \(\dim V^\Gamma = 1 + \dim W^\Gamma\) is either \(1\) or \(2\). We drop the first option, since we are interested in the case where $\dim \mathcal P \geqslant 2$. Thus $\Gamma$ must be cyclic (generated by a unique element and discrete, since it acts properly discontinuously). If \(\Gamma\) is generated by a hyperbolic element, it fixes its own spacelike eigenline in \(W\). Hence \(W^\Gamma\) is that line and \(\dim V^\Gamma = 2\). If \(\Gamma\) is generated by a parabolic element, it fixes a lightlike line in \(W\). Again \(\dim V^\Gamma = 2\).  So, in both cases $\dim\mathcal P=2$.
\end{proof}

We conclude this section by proving Theorem~\ref{T3}. To this end, we first establish the following proposition.

\begin{proposition}\label{prop4}
Assume $\dim \mathcal{P}\geqslant 2$. Let $f_1, f_2$ be two linearly independent static potentials and let $\Sigma=f_1^{-1}(0)\neq\varnothing$ be closed. If $f_2$ does not vanish on $\Sigma$, the following hold:
\begin{enumerate}
    \item[(i)] $(M,g)$ is nice along $\Sigma$;
    \item[(ii)] $(M,g)$ is nice in an open neighborhood of $\Sigma$.
\end{enumerate}
Moreover, if we additionally assume that $R_g\ne 0$, then $\dim \mathcal{P}= 2$.
\end{proposition}

\begin{proof}
Let $P$ be a connected component of $\Sigma$ and $K$ its Gauss curvature. By Lemma~\ref{lem4}, the restriction of $f_2$ to this surface satisfies
$$
\Delta_{P} f_2 = \left(K - \frac{R_g}{2}\right) f_2.
$$
Since $f_2$ is nowhere zero on $P$, after possibly replacing $f_2$ by its negative, we may assume $f_2 > 0$ on $P$. By Proposition~24 in~\cite{ambrozio2017static}, this positivity forces $P$ to have constant Gauss curvature, specifically $K=\dfrac{R_g}{2}$. Lemma~\ref{lem1}(iv) then implies that the Ricci eigenvalues satisfy $R_{11} = R_{22} = \dfrac{R_g}{2}$ and $R_{33} = 0$ along $P$, i.e., $(M,g)$ is nice along $P$. If  $R_g\ne 0$, the metric is not Einstein at any point of $P$, and thus Proposition~\ref{prop2} yields $\dim \mathcal{P}\leqslant 2$, forcing $\dim \mathcal{P}= 2$.

To establish (ii), fix $p \in P$. Since $f_2$ does not vanish on $P$, there exists an open neighborhood $U$ of $p$ where $f_2 \ne 0$. Define $Z = f_1 / f_2$ on $U$; then $Z = 0$ on $P \cap U$. By Lemma~\ref{lem3}(iii), there is an open neighborhood $\Omega \subset U$ of $p$, diffeomorphic to $(-\varepsilon, \varepsilon) \times \Sigma_0$, with $\Sigma_0 \subset P$ a small surface patch containing $p$, such that $Z$ is constant on each slice $\{t\} \times \Sigma_0$, and the metric on $\Omega$ takes the form
$$
g = f_2^2 \, dt^2 + g_0,
$$
where $g_0$ is the induced metric on $\Sigma_0$. From (i), $(\Sigma_0, g_0)$ has constant Gauss curvature $R_g / 2$ and $(M,g)$ is nice along it.

If $R_g\ne 0$, Lemma \ref{lem5} implies that $g$ is nice in $\Omega$.

If $R_g=0$, Lemma \ref{lem6} implies that $g$ is flat in $\Omega$, in particular, it is also nice in there. 
\end{proof}

We now proceed to the proof of Theorem~\ref{T3}.

\begin{proof}
Let $p \in M$ be arbitrary. Since the zero sets of $f_1$ and $f_2$ do not intersect, at least one of $f_1(p)$ or $f_2(p)$ is nonzero. Consequently, in a neighborhood of $p$, one of the potentials is nonvanishing, and the metric $g$ is real-analytic there. Because $M$ is path-connected and, by Proposition~\ref{prop4}(ii), the metric is nice in a neighborhood of every point, it follows that $(M,g)$ is globally nice.

From Proposition~\ref{prop4} we also have $\dim \mathcal{P}=2$ when $R_g\neq 0$.

If $R_g>0$, the Cheeger-Gromoll splitting theorem implies that the universal cover of $(M,g)$ is isometric to $\mathbb{R} \times \mathbb{S}^2$ equipped with the product metric. Hence $(M,g)$ is covered by $\mathbb{R} \times \mathbb{S}^2$ with the product metric, as claimed. The complete classification then follows from Theorem~3.1(iv) in~\cite{kobayashi1982differential} (see also Examples 1 and 2 therein).

If $R_g=0$, a similar argument shows that $(M,g)$ is covered by $\mathbb{R} \times \mathbb{E}^2$ with the product metric, which is isometric to $\mathbb{E}^3$. In $\mathbb{E}^3$, however, the zero set of a nonconstant static potential is not closed. Thus $(M,g)$ must be a nontrivial quotient of $\mathbb{R} \times \mathbb{E}^2$. This case is already covered by the analysis in the proof of Theorem~\ref{T2} and corresponds precisely to the flat manifolds of types $N_3$ and $N_6$ (see Appendix~\ref{appC}), i.e., $\mathbb{R} \times \mathbb{T}^2$, where $\mathbb{T}^2$ is a flat torus, and $\mathbb{R} \times \mathbb{KL}$, where $\mathbb{KL}$ is a flat Klein bottle. In both cases $\dim\mathcal{P}=2$. Manifolds of type $N_5$ are excluded because for them the zero set of a nonconstant static potential is not closed. In all remaining flat cases, the space $\mathcal{P}$ of static potentials is one-dimensional and consists only of constant functions.

\end{proof}

\section{Explicit examples and related questions}

\subsection{Some examples}\label{examples}

In this section, we present two explicit examples of static manifolds that fall outside the scope of the main classification theorems discussed in the Introduction. These manifolds are neither space forms, nor products, and they do not possess parallel Ricci curvature. They are not locally conformally flat. The author is grateful to Lucas Ambrozio for providing these examples.

We begin with the Schwarzschild spacetime of mass $m>0$, given by the Lorentzian metric  
$$
ds^2
= -\Bigl(1 - \frac{2m}{r}\Bigr)\,dt^2
+ \Bigl(1 - \frac{2m}{r}\Bigr)^{-1} dr^2
+ r^2\bigl(d\psi^2 + \sin^2\psi\, d\phi^2\bigr),
\qquad r > 2m.
$$

Applying a \textit{Wick rotation} $t = iT$ and subsequently rescaling the Euclidean time coordinate as $T = \theta$ (so that $\theta$ has period $8\pi m$), we obtain the smooth Riemannian metric  
\begin{align} \label{ESch}
ds^2
= \Bigl(1 - \frac{2m}{r}\Bigr)\,d\theta^2
+ \Bigl(1 - \frac{2m}{r}\Bigr)^{-1} dr^2
+ r^2\bigl(d\psi^2 + \sin^2\psi\, d\phi^2\bigr),
\qquad r > 2m.
\end{align}
In fact, passing to, for example, the Kruskal-Szekeres coordinates in the Schwarzschild spacetime, one can verify that $r=2m$ is only a coordinate singularity and the metric can be extended as a complete metric on $\mathbb{R}^2 \times \mathbb{S}^2$. The same can be done for the metric~\eqref{ESch}. This defines a Ricci-flat (hence Einstein with zero scalar curvature) metric on $\mathbb{R}^2 \times \mathbb{S}^2$.

Consider the associated three-dimensional static metric:  
\begin{align}
\label{Euclidsch}
g = \left(1-\frac{2m}{r}\right)d\theta^2
      + \frac{dr^2}{1-\dfrac{2m}{r}}
      + r^2 d\psi^2 ,
\qquad r > 2m.
\end{align}
In the Kruskal-Szekeres coordinates the above metric takes the form
\[
g = \frac{32m^3}{r}\, e^{-r/(2m)}\bigl(dX^2 + dY^2\bigr) + r^2\theta^2,
\]
where
\[
X^2 + Y^2 = \left(\frac{r}{2m}-1\right)e^{r/(2M)}.
\]
This shows that the metric~\eqref{Euclidsch} can be defined on $\mathbb{R}^2 \times \mathbb{S}^1$ as a complete metric. This metric is static: The associated static potentials are  
$$
V_1(r,\psi) = r\sin\psi, \qquad V_2(r,\psi) = r\cos\psi,
$$
which are linearly independent. In fact, they form a basis of the space $\mathcal{P}$ of static potentials, so $\dim\mathcal{P} = 2$, as we explain below (see~Proposition~\ref{specstatdim}).

A similar construction can be carried out using the Schwarzschild--anti-de Sitter spacetime. After Wick rotation, one obtains the static Riemannian metric  
$$
g = \left(1-\dfrac{2m}{r}+r^2\right)\,d\theta^2 + \frac{dr^2}{1-\dfrac{2m}{r}+r^2} + r^2 d\psi^2,
$$
which has negative scalar curvature and again admits two linearly independent static potentials $V_1 = r\sin\psi$ and $V_2 = r\cos\psi$, yielding $\dim\mathcal{P} = 2$ (see~Proposition~\ref{specstatdim} again).

\subsection{General form and geometric interpretation}\label{meaning}

Both examples in the previous section belong to a broader family of static metrics of the form  
\begin{align}
\label{form}
g=\left(k_1+\dfrac{k_2}{r}+k_3r^2\right)dx^2+\dfrac{dr^2}{k_1+\dfrac{k_2}{r}+k_3r^2}+r^2d\phi^2,
\end{align}
where $k_1, k_2, k_3$ are real constants. According to \cite[Proposition 10]{tod2000spatial}, any static metric that admits two functionally independent static potentials is locally of this form in suitable coordinates $(x, r, \phi)$. The geometric meaning of the constants $k_i$ is as follows:

\begin{itemize}
    \item The coefficient $k_2$ measures the deviation from being Einstein. Indeed, by \cite[Proposition 9]{tod2000spatial}, the Ricci tensor in these coordinates satisfies  
    $$
    \label{ricci}
    R_{11}=R_{22}=-\dfrac{k_2}{2r^3}+\dfrac{R_g}{3}, \quad R_{33}=\dfrac{k_2}{r^3}+\dfrac{R_g}{3}.
    $$

    \item The coefficient $k_3$ is directly related to scalar curvature:  
    $$
    k_3 = -\dfrac{R_g}{6},
    $$
    as shown in the proof of \cite[Proposition 10]{tod2000spatial}. 
    
    \item The constant $k_1$ governs the static potentials (see the proof of  \cite[Proposition 10]{tod2000spatial}, formula (38)). Specifically, static potentials of the form $f = r H(\phi)$ arise when $H$ satisfies the ordinary differential equation  
    $$
    H_{\phi\phi} = -k_1 H.
    $$
    The solution space of this equation is two-dimensional, yielding exactly two linearly independent static potentials. Moreover, any static potential $f$ satisfies the eigenvalue equation  
    $$
    \Delta_g f = 3k_3 f.
    $$
\end{itemize}

\subsection{Dimension of the space of static potentials}

We now establish a precise statement about the dimension of $\mathcal{P}$ for metrics of the form \eqref{form}.

\begin{proposition}\label{specstatdim}
Suppose a manifold $M$ admits coordinates $(x, r, \phi)$ and a metric of the form
$$
g = \left(k_1 + \frac{k_2}{r} + k_3 r^2\right)dx^2 
    + \frac{dr^2}{k_1 + \dfrac{k_2}{r} + k_3 r^2} 
    + r^2 d\phi^2,
$$
with $k_2 \neq 0$. If either ($\phi$ is $2\pi$-periodic and $k_1 > 0$) or ($\phi$ is non-periodic and $k_1 \leqslant 0$), then the space $\mathcal{P}$ of static potentials is two-dimensional. Specifically, $\mathcal{P}$ is spanned by:
\begin{itemize}
    \item $r\cos\phi$ and $r\sin\phi$, in the periodic case ($k_1 > 0$);
    \item $r$ and $r\phi$, in the non-periodic case ($k_1 = 0$);
        \item $r\cosh\phi$ and $r\sinh\phi$, in the non-periodic case ($k_1 < 0$).
\end{itemize}
\end{proposition}

\begin{proof}
We consider three cases based on the sign of $k_1$ (with $k_2 \neq 0$ in all cases). In each case the equation $H_{\phi\phi} = -k_1 H$ has two linearly independent solutions, giving two independent static potentials of the form $r H(\phi)$:
\begin{itemize}
\item If $k_1 > 0$, we may set $k_1 = 1$ for simplicity; then $H_{1,2}(\phi) = \cos\phi,\;\sin\phi$.
\item If $k_1 = 0$, then $H_{1,2}(\phi) = 1,\;\phi$.
\item If $k_1 < 0$, we may set $k_1 = -1$ for simplicity; then $H_{1,2}(\phi) = \cosh\phi,\;\sinh\phi$.
\end{itemize}
For any $\phi_0$ in the domain ($\mathbb{R}$ or $[0,2\pi)$), there exists a nontrivial linear combination of the two basis functions that vanishes at $\phi_0$. Indeed, we can solve respectively
$a\cos\phi_0 + b\sin\phi_0 = 0$, $a + b\phi_0 = 0$, or $a\cosh\phi_0 + b\sinh\phi_0 = 0$ with not both $a,b$ zero. Fix $\phi_0$ and let $u$ be such a combination vanishing at $\phi_0$. Then $ru$ is also a static potential; its zero set is $\{\phi = \phi_0\}$.

Assume, for contradiction, that there exists a third static potential $f$ linearly independent of the two found above ($rH_1$ and $rH_2$). Apply Lemma~\ref{lem4} to $ru$ and $f$ to obtain
$$
\Bigl(K - \frac{R_g}{6}\Bigr) f^3 = \text{const} \quad \text{on } \phi = \phi_0.
$$
By Lemma~\ref{lem1}, the Gauss curvature of the surface $\phi = \text{const}$ is
$$
K = -R_{33} + \frac{R_g}{2} = -\frac{k_2}{r^3} - k_3,
$$
hence
$$
K - \frac{R_g}{6} = -\frac{k_2}{r^3}.
$$
Since $k_2 \neq 0$, this expression depends nontrivially on $r$. Therefore on $\phi = \phi_0$, $f^3$ is proportional to $r^3$ on that surface, i.e., $f=Cr$ for $\phi=\phi_0$. Since $\phi_0$ was chosen arbitrarily, we conclude that, at least locally,
$$
f(r, \phi) = C(\phi) r.
$$

Since $f$ is independent of the coordinate $x$, the Laplacian acts on it as
$$
\Delta_g f = \left(k_1 + \frac{k_2}{r} + k_3 r^2\right) \frac{\partial^2 f}{\partial r^2} 
+ \left(\frac{k_1}{r} + 3k_3 r\right) \frac{\partial f}{\partial r} 
+ \frac{1}{r^2} \frac{\partial^2 f}{\partial \phi^2}.
$$
With $f = C(\phi) r$ we have $f_r = C(\phi)$ and $f_{\phi\phi} = C''(\phi) r$, so
$$
\Delta_g f = \frac{1}{r}\bigl(k_1 C(\phi) + C''(\phi)\bigr) + 3k_3 r\,C(\phi).
$$
The static condition $\Delta_g f = 3k_3 f = 3k_3 C(\phi) r$ yields
$$
C''(\phi) =-k_1 C(\phi).
$$
We see that the function $C$ satisfies the same ODE as $H$, and $f=Cr$. Hence $f$ depends linearly on $rH_1$ and $rH_2$, contradicting the assumption.

Therefore no such third static potential exists, and the space $\mathcal{P}$ is two-dimensional, spanned by the two potentials listed in each case. 
\end{proof}

\begin{remark}
In the case where $k_2=0$ one may has $\dim \mathcal P=4,2$ or $1$. Indeed, consider the metric 
$$
g = r^2dx^2 
    + \frac{dr^2}{r^2} 
    + r^2 d\phi^2,
$$ 
for which $k_1=k_2=0$ and $k_3=1$ (case 5 in the below classification). Using the change of variables $r=1/z$, is not difficult to see that the metric takes the form
$$
g=\dfrac{dx^2+d\phi^2+dz^2}{z^2},
$$
which is exactly the standard metric on hyperbolic 3-space $\mathbb H^3$ in the upper half-space model if $x,\phi \in \mathbb R$. As we know, $\dim \mathcal P=4$: $\mathcal P$ is spanned by 
\begin{align*}
x_0 &= \frac{1 + x^{2} + \phi^{2} + z^{2}}{2z}, \\
x_1 &= \frac{x}{z}, \\
x_2 &= \frac{\phi}{z}, \\
x_3 &= \frac{1 - x^{2} - \phi^{2} - z^{2}}{2z}
\end{align*}
(the coordinate functions in the hyperboloid model). However, if $x$ is periodic and $\phi\in \mathbb R$ or $\phi$ is periodic and $x\in \mathbb R$, we obtain a complete hyperbolic manifold with a rank-1 cusp of infinite volume, for which $\dim\mathcal P=2$ ($x_0+x_3$ and $x_2$ or $x_1$, respectively, descend to the quotient space). Finally, if both $x$ and $\phi$ are periodic, we obtain a complete finite-volume hyperbolic 3-manifold with cusp and $\dim\mathcal P=1$ (only $x_0+x_3$ descends to the quotient space). This shows that in the case where $k_1=k_2=0$ one may has $\dim \mathcal P=2$ or $1$ even if $\phi$ is periodic.
\end{remark}

\subsection{Completeness questions}\label{compquest} Since we consider only complete metrics in this paper, the metric induced on the zero set of any static potential is also complete, as the zero set is totally geodesic. Furthermore, the completeness of the static metric imposes constraints on the coefficients $k_1$, $k_2$, and $k_3$ in~\eqref{form}. Specifically, we seek coefficients for which the function $f(r):=k_1+\dfrac{k_2}{r}+k_3r^2$ admits positive roots. While $f$ may have either one or two positive roots, or none at all, a straightforward analysis yields the table below. We explain the derivation of the table in Section~\ref{table}.

\medskip

\begin{tabular}{|c|c|}
\hline
\textbf{Conditions on \( (k_1, k_2, k_3) \)} & \textbf{\# distinct positive roots} \\
\hline
\(k_3 \neq 0,\; k_2 \neq 0,\; k_3k_2 > 0,\; k_3k_1 < 0,\; 27|k_3|k_2^2 < 4|k_1|^3\) & 2 \\
\hline
\(k_3 \neq 0,\; k_2 \neq 0,\; k_3k_2 > 0,\; k_3k_1 < 0,\; 27|k_3|k_2^2 = 4|k_1|^3\) & 1 (double) \\
\hline
\(k_3 = 0,\; k_2 \neq 0,\; k_1k_2 < 0\) & 1 \\
\hline
\(k_3 \neq 0,\; k_2 = 0,\; k_1k_3 < 0\) & 1 \\
\hline
\(k_3 \neq 0,\; k_2 \neq 0,\; k_1k_3 > 0,\; k_2k_3 < 0\) & 1 \\
\hline
\(k_3 \neq 0,\; k_2 \neq 0,\; k_1 = 0,\; k_2k_3 < 0\) & 1 \\
\hline
\(k_3 \neq 0,\; k_2 \neq 0,\; k_1k_3 < 0,\; k_2k_3 < 0,\; 4k_1^3+27k_3k_2^2 < 0\) & 1 \\
\hline
\(k_3 \neq 0,\; k_2 \neq 0,\; k_1k_3 < 0,\; k_2k_3 < 0,\; 4k_1^3+27k_3k_2^2 = 0\) & 1 \\
\hline
\(k_3 \neq 0,\; k_2 \neq 0,\; k_1k_3 < 0,\; k_2k_3 < 0,\; 4k_1^3+27k_3k_2^2 > 0\) & 1 \\
\hline
\(k_3 = 0,\; k_2 = 0,\; k_1 \neq 0\) & 0 \\
\hline
\(k_3 = 0,\; k_2 \neq 0,\; k_1k_2 \geqslant 0\) & 0 \\
\hline
\(k_3 \neq 0,\; k_2 = 0,\; k_1k_3 \geqslant 0\) & 0 \\
\hline
\(k_3 \neq 0,\; k_2 \neq 0,\; k_1k_3 > 0,\; k_2k_3 \geqslant 0\) & 0 \\
\hline
\(k_3 \neq 0,\; k_2 \neq 0,\; k_1 = 0,\; k_2k_3 > 0\) & 0 \\
\hline
\(k_3 \neq 0,\; k_2 \neq 0,\; k_1k_3 < 0,\; k_2k_3 > 0,\; 27|k_3|k_2^2 > 4|k_1|^3\) & 0 \\
\hline
\end{tabular}

\medskip

Further analysis yields that $g$ of the form~\eqref{form} is complete in the following cases (for an explanation see Section~\ref{metriccomp}).

\begin{enumerate}
  \item \textbf{No roots (smooth on $[0,\infty)$).} \\
        Parameters: $k_2 = 0$, $k_1 > 0$, $k_3 \geqslant 0$. \\
        Domain: $r \in [0,\infty)$. \\
        Periods: $\phi$ with period $2\pi/\sqrt{k_1}$; $x$ can be non-compact or have any period $L>0$. \\
        Geometry: $\mathbb{E}^3$ when $k_3 = 0$, hyperbolic space $\mathbb{H}^3$ when $k_3 > 0$.

  \item \textbf{One simple root, outer domain $[r_0,\infty)$ (non-compact).} \\
        A simple root $r_0>0$ with $f'(r_0)>0$ occurs in the following cases:
        \begin{itemize}
          \item $k_3 > 0$, $k_2 < 0$ (any $k_1$);
          \item $k_3 > 0$, $k_2 = 0$, $k_1 < 0$  (then $r_0 = \sqrt{-k_1/k_3}$);
          \item $k_3 > 0$, $k_2 > 0$ and $k_1 < -3\left(\dfrac{k_2^2 k_3}{4}\right)^{1/3}$ (the larger of two roots);
          \item $k_3 = 0$, $k_1 > 0$, $k_2 < 0$  (then $r_0 = -k_2/k_1$).
        \end{itemize}
        Domain: $r \in [r_0,\infty)$. \\
        Periods: $x$ must have period $4\pi/f'(r_0)$; $\phi$ can be non-compact in some cases or have any period. \\
        Geometry: non-compact manifold with a smooth bolt at $r=r_0$.

  \item \textbf{One simple root, inner domain $[0,r_0]$ (compact).} \\
        Parameters: $k_2 = 0$, $k_1 > 0$, $k_3 < 0$. \\
        Root: $r_0 = \sqrt{-k_1/k_3}$, with $f'(r_0) = -2\sqrt{-k_1k_3} < 0$. \\
        Domain: $r \in [0,r_0]$. \\
        Periods: $\phi$ with period $2\pi/\sqrt{k_1}$; $x$ with period $4\pi/|f'(r_0)| = 2\pi/\sqrt{-k_1k_3}$. \\
        Geometry: compact 3-manifold (squashed 3-sphere) with a smooth axis at $r=0$ and a smooth bolt at $r=r_0$.

  \item \textbf{Two simple roots, bounded domain $[r_1,r_2]$ (compact, fine-tuned).} \\
        Parameters: $k_3 < 0$, $k_1 > 0$, $k_2 < 0$ and $27|k_3|k_2^2 < 4|k_1|^3$ (so two positive roots exist). \\
        Fine-tuning condition for smooth bolts at both ends:
        \[
        |f'(r_1)| = |f'(r_2)| \quad \Longleftrightarrow \quad r_1 r_2 = -\frac{k_1}{3k_3}.
        \]
        Domain: $r \in [r_1,r_2]$. \\
        Periods: $x$ with period $4\pi/|f'(r_1)| = 4\pi/|f'(r_2)|$; $\phi$ any period. \\
        Geometry: compact 3-manifold with two smooth bolts.

  \item \textbf{Double root -- cusp (non-compact).} \\
        A double root $r_0\geqslant 0$ (where $f(r_0)=f'(r_0)=0$) yields a boundary at infinite distance. These occur for:
        \begin{itemize}
          \item $k_3 > 0$, $k_2 = 0$, $k_1 = 0$: double root at $r=0$, domain $(0,\infty)$;
          \item $k_3 > 0$, $k_2 > 0$ and $k_1 = -3\left(\dfrac{k_2^2 k_3}{4}\right)^{1/3}$: double root at $r_0 = \left(\dfrac{k_2}{2k_3}\right)^{1/3}$, domain $[r_0,\infty)$.
        \end{itemize}
        No period conditions are needed; the metric is complete with a cusp at the boundary.
\end{enumerate}

\section{Static manifolds with boundary}

In this section, we focus on the case of static manifolds with boundary. We start with the following lemma, which is repeatedly used in our subsequent analysis.

\begin{lemma}\label{lemomega}
A connected domain $\Omega$ is a static manifold with boundary in the following spaces if and only if:
\begin{enumerate}[(i)]
    \item In $\mathbb{S}^3$: $\Omega$ is a spherical cap or a spherical ring (the region between two concentric geodesic spheres).
    \item In $\mathbb{E}^3$: $\Omega$ is a half-space, a ball, the exterior of a ball, a slab between parallel planes, or a spherical shell (the region between two concentric spheres).
    \item In $\mathbb{R} \times \mathbb{S}^2$: $\Omega = I \times \mathbb{S}^2$, where $I \subset \mathbb{R}$ is a closed interval (finite or infinite). In this case the space of statc potentials is one-dimensional.
\end{enumerate}
\end{lemma}

\begin{proof}
$(i)$ By Theorem 4.3 in~\cite{sheng2025static}, any static manifold with boundary on the sphere is an intersection of spherical caps. In the case where $\Omega$ is the intersection of two spherical caps, we show the boundary spheres must be concentric. Proposition 1(a.2) in~\cite{cruz2023static} implies that the zero set of any static potential intersects the boundary orthogonally. Choose coordinates on $\mathbb{S}^3\subset\mathbb{E}^4$ so that one cap is centered at the north pole; then the space of static potentials is $\mathcal{P}=\operatorname{span}\{x_1,x_2,x_3\}$. If the two boundary components were not concentric, one could construct a potential whose zero set fails to be orthogonal to one of the boundary components, contradicting the orthogonality condition. Hence the spheres are concentric. This argument also implies that no third boundary component for $\Omega$ is possible.

$(ii)$ Theorem 4.7 in~\cite{sheng2025static} states that the boundary components of a static manifold $\Omega\subset\mathbb{R}^3$ are either planes or round spheres. We consider two cases.

\medskip

\emph{Case 1: A plane boundary component.} Suppose $x_3=c$ is a boundary component. From the discussion preceding Theorem 4.7 in~\cite{sheng2025static}, the space of static potentials for the half-space $\{x_3>c\}$ is spanned by $1,x_1,x_2$. Any non-trivial potential has zero set (if non-empty) orthogonal to the boundary. If $\Omega$ has another boundary component, it must be either a plane or a sphere.  
\begin{itemize}
    \item If it is a plane, it must be parallel to $x_3=c$; otherwise the planes would intersect, giving a non-smooth boundary and violating orthogonality. Thus we obtain a slab between two parallel planes.
    \item If it were a sphere, one can always construct a static potential whose zero set does not intersect the sphere orthogonally, which is impossible.
\end{itemize}
Hence, when a plane component exists, $\Omega$ is necessarily a slab. Similar reasoning implies that no third boundary component for $\Omega$ is possible.

\medskip

\emph{Case 2: A spherical boundary component.} If a round sphere is a boundary component, a similar argument using Proposition 4.2 in~\cite{ho2020deformation} shows that $\Omega$ must be a ball, the complement of a ball, or a spherical shell.

$(iii)$ The first equation of~\eqref{Robin} implies that any static potential must be of the form $C_1 \sin r + C_2 \cos r$, where $r \in \mathbb{R}$ and $C_1, C_2$ are constants (see e.g. Example 1 in~\cite{kobayashi1982differential}). Let $\Omega$ be a domain in $\mathbb{R} \times \mathbb{S}^2$. The Robin boundary condition implies that $\dim \mathcal{P} = 1$. Indeed, let $\nu$ denote the outward unit normal to $\partial \Omega$. It has the form $\nu = \phi(x,r) \partial_r + X$, where $x \in \mathbb{S}^2$ and $X$ is a vector field orthogonal to $\partial_r$ at each point $(x,r) \in \partial \Omega$. In particular, $X$ has no component in the $\partial_r$ direction. If $\Omega$ is a static manifold with boundary, then the second equation of~\eqref{Robin} must hold on $\partial \Omega$, namely
$$
\phi(x,r)(C_1 \cos r - C_2 \sin r) = \frac{H_{\partial \Omega}}{2} \bigl(C_1 \sin r + C_2 \cos r\bigr),
$$
where $H_{\partial \Omega} = \mathrm{const}$. This equality forces $\phi(x,r) = \phi(r)$, i.e., $\phi$ is independent of $x$. Solving the above equation yields $r = r_0 = \mathrm{const}$, which describes $\partial \Omega$. In other words, each connected component of $\partial \Omega$ is of the form $\{r_0\} \times \mathbb{S}^2$. Furthermore, in this setting $H_{\partial \Omega}=0$ and $r_0$ satisfies
$$
C_1 \cos r_0 - C_2 \sin r_0 = 0.
$$
This implies $\dim\mathcal P=1$.
\end{proof}

We use this lemma to analyze the relative position of the zero sets of linearly independent static potentials. Our first result in this direction is as follows.

\begin{proposition}\label{propstat}
Suppose that $(M^3,g)$ is a compact static manifold with boundary and $\dim\mathcal P\geqslant 2$. Let $f_1,f_2\in \mathcal P$ be linearly independent and $f^{-1}_1(0)\cap\partial M\neq\varnothing$. Then $f^{-1}_2(0)\cap \partial M\neq \varnothing$.
\end{proposition}

\begin{proof}
Since both $f_1$ and $f_2$ satisfy the Robin boundary condition~\eqref{Robin}, they belong to the same eigenspace of the associated eigenvalue problem. Suppose, for contradiction, that $f_2$ does not vanish anywhere in $M$. Then $f_2$ is a first eigenfunction, and $f_1$ must also be a first eigenfunction. However, a first eigenfunction cannot vanish in $M$, contradicting the assumption that $f_1$ has a nonempty zero set. Therefore, the zero set of $f_2$ is nonempty.

Assume further that $f_2^{-1}(0) \cap \partial M = \varnothing$. Then, by Proposition~1 (a.1) in~\cite{cruz2023static}, every connected component of $f_2^{-1}(0)$ is a closed, totally geodesic surface in the interior of $M$. Now suppose that the zero sets of $f_1$ and $f_2$ are disjoint. In this case, the argument used in the proof of Theorem~\ref{T3} applies, and we conclude that the Riemannian manifold $(M,g)$ is nice -- that is, its Ricci tensor has eigenvalues $\left(\dfrac{R_g}{2}, \dfrac{R_g}{2}, 0\right)$.

Let $P_1$ be a connected component of $f_1^{-1}(0)$ that intersects the boundary $\partial M$. Then, by Proposition~1 (a.2) in~\cite{cruz2023static}, $P_1$ is a free boundary totally geodesic surface. Consider the restriction of $f_2$ to $P_1$. By Lemma~\ref{lem4} and the free boundary condition, this restriction satisfies the following boundary value problem:
$$
\begin{cases}
\Delta_{P_1} f_2 = \Bigl(K_1 - \dfrac{R_g}{2}\Bigr) f_2 & \text{in } P_1, \\
\dfrac{\partial f_2}{\partial \nu} = \dfrac{H_{\partial M}}{2} f_2 & \text{on } \partial P_1,
\end{cases}
$$
where $K_1$ denotes the Gauss curvature of $P_1$, $\nu$ is the outward unit conormal to $\partial P_1$. However, since $(M,g)$ is nice, Lemma~\ref{lem1} implies that the Gauss curvature of any connected component of $f_1^{-1}(0)$ equals $\dfrac{R_g}{2}$. Consequently, the system simplifies to
$$
\begin{cases}
\Delta_{P_1} f_2 = 0 & \text{in } P_1, \\
\dfrac{\partial f_2}{\partial \nu} = \dfrac{H_{\partial M}}{2} f_2 & \text{on } \partial P_1.
\end{cases}
$$

Since $f_2$ does not vanish on $P_1$ (by the disjointness assumption), its restriction to $P_1$ is a nontrivial harmonic function with constant sign. By Hopf's lemma, we must have $H_{\partial M} = 0 \quad \text{and} \quad f_2 \equiv \text{const on } P_1$. (This also follows from the observation that the restriction of $f_2$ to $P_1$ is a first Steklov eigenfunction.)

Now, by the result of~\cite{cruz2019prescribing} if $H_{\partial M} = 0 $, the scalar curvature must satisfy $R_g \geqslant 0$. Since $(M,g)$ is nice and $R_g \geqslant 0$, the Cheeger-Gromol splitting theorem implies that the universal Riemannian cover of $(M,g)$ is isometric to a domain $\tilde{M}$ in either $\mathbb{R} \times \mathbb{E}^2$ with the product metric (this is, clearly, isometric to $\mathbb E^3$), if $R_g = 0$, or $\mathbb{R} \times \mathbb{S}^2$ with the product metric, if $R_g > 0$. In both cases, the boundary $\partial M$ is totally geodesic (as $H_{\partial M} = 0 $), so $\partial \tilde{M}$ is also totally geodesic. However, Lemma~\ref{lemomega} shows that such a configuration cannot occur: in the case of $\mathbb{R} \times \mathbb{E}^2$ ($\tilde{M}$  is a half-space or a slab between two parallel planes), the zero set of any non-constant static potential intersects the boundary; in the case of $\mathbb{R} \times \mathbb{S}^2$, the zero set of any static potential does not intersect the boundary.

We now assume that $f_2^{-1}(0) \cap \partial M = \varnothing$, and that there exist a closed connected component $P_2 \subset f_2^{-1}(0)$ and a connected component with boundary $P_1 \subset f_1^{-1}(0)$ such that $P_1 \cap P_2 \neq \varnothing$. Since $P_2$ is disjoint from $\partial M$, the intersection $P_1 \cap P_2$ lies entirely in the interior of $M$. In particular, the restriction of $f_2$ to $P_1$ changes sign on $P_1$ but does not vanish on $\partial P_1$.

As in the proof of Theorem~\ref{T1}, this situation implies that $(M,g)$ is a space form, and both $P_1$ and $P_2$ have Gauss curvature equal to $\dfrac{R_g}{6}$. We rescale the metric so that $R_g = 6\epsilon$, where $\epsilon \in \{-1,0,1\}$.

By Lemma~\ref{lem4}, the restriction of $f_2$ to $P_1$ satisfies
$$
\begin{cases}
\mathrm{Hess}_{P_1} f_2 = -\epsilon f_2 \, g_{P_1} & \text{in } P_1, \\
\dfrac{\partial f_2}{\partial \nu} = \dfrac{H_{\partial M}}{2} f_2 & \text{on } \partial P_1,
\end{cases}
$$
and the restriction of $f_1$ to $P_2$ satisfies
$$
\mathrm{Hess}_{P_2} f_1 = -\epsilon f_1 \, g_{P_2}.
$$
Taking the trace of the latter equation yields
$$
\Delta_{P_2} f_1 = -2\epsilon f_1.
$$

If $\epsilon = -1$, then $f_1$ satisfies $\Delta_{P_2} f_1 = 2 f_1$ on the closed surface $P_2$. The only solution is $f_1 \equiv 0$ on $P_2$, which contradicts the fact that the zero sets of distinct eigenfunctions intersect transversally (and hence $f_1$ cannot vanish identically on a component of $f_2^{-1}(0)$). Therefore, $\epsilon \in \{0,1\}$.

Suppose first that $\epsilon = 0$. Consequently, the restriction of $f_2$ to $P_1$ satisfies 
$$
\begin{cases}
\Delta_{P_1} f_2 = 0 & \text{in } P_1, \\
\dfrac{\partial f_2}{\partial \nu} = \dfrac{H_{\partial M}}{2} f_2 & \text{on } \partial P_1.
\end{cases}
$$
Since $f_2$ does not vanish on $\partial P_1$, Hopf's lemma implies that $H_{\partial M} = 0$, and hence $f_2$ is constant on $P_1$. This contradicts the fact that $f_2$ changes sign on $P_1$. Hence, the case $\epsilon = 0$ is impossible.

We conclude that $\epsilon = 1$. Let $(\tilde{M}, \tilde{g})$ be the universal Riemannian cover of $(M,g)$. Since $(M,g)$ is a space form with positive scalar curvature, the cover is isometric to a domain in the standard sphere $\mathbb{S}^3$, bounded by round spheres (the boundary components of a static manifold with boundary are umbilic). The closed surface $P_2 \subset M$ lifts to a closed totally geodesic surface in $\tilde{M}$. This is a connected component of the zero set of the lift of $f_2$ which is a static potential on $(\tilde{M}, \tilde{g})$. But by Lemma~\ref{lemomega}, $(\tilde{M}, \tilde{g})$ is either spherical cap or a spherical ring. The zero set of any static potential in both cases is connected and intersects the boundary, i.e. there are no closed connected components. This is a contradiction.

Finally, consider the possibility that both $P_1\subset f^{-1}_1(0)$ and $P_2\subset f^{-1}_2(0)$ are closed and intersect and $f^{-1}_2(0)$ does not intersect any free boundary component of $f^{-1}_1(0)$. Then the same argument applies: $(M,g)$ is a space form with $R_g = 6\epsilon$, $\epsilon \in \{0,1\}$, and the above analysis shows both cases lead to contradictions. Indeed, if $\epsilon = 0$, we again obtain a nonvanishing solution to
$$
\begin{cases}
\mathrm{Hess}_{P} f_2 = 0 & \text{in } P, \\
\dfrac{\partial f_2}{\partial \nu} = \dfrac{H_{\partial M}}{2} f_2 & \text{on } \partial P,
\end{cases}
$$
for some component $P \subset f_1^{-1}(0)$ with boundary, which forces $H_{\partial M} = 0$ . Thus, $\partial M$ is totally geodesic and $(M,g)$ is flat (as $R_g=6\epsilon=0$). It follows that $(M,g)$ admits a Riemannian covering by a domain $\tilde{M} \subset \mathbb{E}^3$. Each boundary component of $M$ lifts to a totally geodesic surface in $\mathbb{E}^3$, i.e., a Euclidean plane. By Lemma~\ref{lemomega}, $\tilde{M}$ is either a half-space (one boundary component) or a slab between two parallel planes (two boundary components). In either case, the lift of $f_2^{-1}(0)$ is a totally geodesic surface which does not intersect the boundary. However, the zero set of any static potential in these domains is connected and intersects the boundary. We arrive to a contradiction.

 If $\epsilon = 1$, the universal cover is a domain in $\mathbb{S}^3$, and Lemma~\ref{lemomega} again rules out the existence of the required closed surfaces (no static potential has zero set intersecting the boundary).

Therefore, every connected component of $f_2^{-1}(0)$ must intersect $\partial M$.
\end{proof}

As a byproduct of this proof, we obtain the following corollary.

\begin{corollary}\label{remff}
Suppose that $(M^3,g)$ is a compact static manifold with boundary and $\dim\mathcal P\geqslant 2$. Let $f_1,f_2\in \mathcal P$ be linearly independent. If one connected component of $f^{-1}_1(0)$ intersects the boundary, then all connected components of $f^{-1}_1(0)$ and $f^{-1}_2(0)$ also intersect the boundary.
\end{corollary}

In our subsequent analysis, we will primarily rely on this corollary. However, Proposition~\ref{propstat} may be of independent interest.

Our aim now is to prove Theorem~\ref{thmbstat}. To this end, we first prove the following lemma. 

\begin{lemma}\label{propprep}
Suppose that $(M^3,g)$ is a compact static manifold with boundary and $\dim\mathcal P\geqslant 2$. Let $f_1,f_2\in \mathcal P$ be linearly independent and $f^{-1}_1(0)\cap\partial M=\varnothing$. Then, $f^{-1}_1(0)\cap f^{-1}_2(0)= \varnothing$ and $(M,g)$ is nice.
\end{lemma}

\begin{proof}
Assume, for contradiction, that the zero sets intersect. Thus, there exist connected components $P_1 \subset f_1^{-1}(0)$ and $P_2 \subset f_2^{-1}(0)$ such that $P_1 \cap P_2 \neq \varnothing$. Applying the argument from the proof of Theorem~\ref{T1} to this intersecting pair of level sets, we conclude that $(M,g)$ is a space form with scalar curvature $R_g = 6\epsilon,~\epsilon \in \{-1,0,1\}$, after a suitable rescaling of the metric. 

By Proposition 1 (a.1) in~\cite{cruz2023static}, $P_1$ is closed. Restricting $f_2$ to $P_1$, we recall from the previous analysis that the induced Laplacian satisfies
$$
\Delta_{P_1} f_2 = -2\epsilon f_2.
$$
Hence $\epsilon \in \{0,1\}$.

Consider the universal cover $\widetilde{M}$ of $M$, equipped with the lifted metric $\tilde{g}$. Because $(M,g)$ has constant curvature with $R_g = 6\varepsilon$ and admits static potentials, $(\widetilde{M},\tilde{g})$ is isometric to a domain $\Omega$ in either Euclidean space $\mathbb{E}^3$ (if $\varepsilon = 0$) or the round sphere $\mathbb{S}^3$ (if $\epsilon = 1$). Moreover, $\Omega$ inherits the structure of static manifold with boundary, and the lifts $\tilde{f}_1, \tilde{f}_2$ are static potentials on $\Omega$. It is known that such domains $\Omega$ must be one of the domains in $(i)$ and $(ii)$ in Lemma~\ref{lemomega}, respectively. In all these cases, the zero set of any nontrivial static potential (if nonempty) necessarily intersects the boundary $\partial \Omega$. Consequently, the zero sets of $\tilde{f}_1$ and $\tilde{f}_2$ intersect $\partial \Omega$. Projecting back to $M$, this implies that $f_1^{-1}(0) \cap \partial M \neq \varnothing$ and $f_2^{-1}(0) \cap \partial M \neq \varnothing$, contradicting the assumptions.

This contradiction shows that our assumption was false. Therefore, $f_1^{-1}(0) \cap f_2^{-1}(0) = \varnothing$, as stated. As in the proof of Theorem~\ref{T3} this implies that $(M,g)$ is nice.
\end{proof}

We now ready to prove Theorem~\ref{thmbstat}.

\begin{proof}[Proof of Theorem~\ref{thmbstat}.]
Assume $f_1, f_2 \in \mathcal{P}$ are linearly independent and $f_1^{-1}(0) \cap \partial M = \varnothing$. By Lemma~\ref{propprep}, the zero sets are disjoint ($f_1^{-1}(0) \cap f_2^{-1}(0) = \varnothing$) and $(M,g)$ is nice. Moreover, Corollary~\ref{remff} implies that $f_2^{-1}(0) \cap \partial M = \varnothing$. According to Proposition 1(a.1) in~\cite{cruz2023static}, $f_1^{-1}(0)$ and $f_2^{-1}(0)$ are closed surfaces. Restricting $f_2$ to a connected component of $f_1^{-1}(0)$ yields $R_g \geqslant 0$, following the argument established in the preceding proofs.

Since $R_g \geqslant 0$, the nice manifold $(M,g)$ admits a Riemannian covering by a domain in $\mathbb{R} \times \mathbb{E}^2$ (if $R_g = 0$) or $\mathbb{R} \times \mathbb{S}^2$ (if $R_g > 0$). In the former case ($\mathbb{R} \times \mathbb{E}^2$), any domain $\Omega$ that is a static manifold with boundary must be of the type described in Lemma~\ref{lemomega}(ii). In all such instances, the zero set of any nontrivial static potential necessarily intersects $\partial \Omega$, which contradicts the assumption that $f_1^{-1}(0) \cap \partial M = \varnothing$. In the latter case ($\mathbb{R} \times \mathbb{S}^2$), Lemma~\ref{lemomega}(iii) implies that the space of static potentials is one-dimensional, contradicting the linear independence of $f_1$ and $f_2$.
\end{proof}

\section{Appendix A} \label{appC}

All complete flat 3-manifolds are quotients $\mathbb{R}^3/\Gamma$, where 
$\Gamma \subset \operatorname{Isom}(\mathbb{E}^3) = \mathbb{R}^3 \rtimes O(3)$ 
is a discrete, torsion-free subgroup. Below we list each manifold $N_i$, give explicit generators for $\Gamma$, and describe the topological type of the quotient $M = \mathbb{R}^3 / \Gamma$. This table is based on~\cite[Section 3.5]{wolf1967spaces}, \cite{conway2003describing}, and~\cite[Section II, the chapter by M.Demia\'nski]{sanchez2012early}.

\bigskip

\subsection*{Isometric classification of complete flat 3-manifolds \(\mathbb{R}^3/\Gamma\)}

\footnotesize
\renewcommand{\arraystretch}{1.2}
\begin{longtable}{|c|p{0.7\linewidth}|p{0.2\linewidth}|}
\hline
\textbf{Label} & \textbf{Generators \(\Gamma\) (with parameters)} & \textbf{Quotient \(M\)} \\
\hline
\endhead
\hline
\multicolumn{3}{r}{\small\textit{Continued on next page}} \\
\endfoot
\hline
\endlastfoot
\(N_1\) & -- & \(\mathbb{R}^3\) \\
\hline
\(N_2\) & \(t(x,y,z) = (x+L,\; y,\; z),\quad L>0\) & \(\mathbb{E}^2\times \mathbb S^1\) \\
\hline
\(N_3\) & \(\begin{aligned} t_1(x,y,z) &= (x,\; y+\alpha,\; z)\\ t_2(x,y,z) &= (x,\; y+\beta,\; z+\gamma) \end{aligned}\) with \(\alpha,\gamma>0,\ \beta\in\mathbb{R}\) & \(\mathbb{R}\times \mathbb T^2\) \\
\hline
\(N_4\) & \(\begin{aligned} t_1(x,y,z) &= (x+a,\; y,\; z)\\ t_2(x,y,z) &= (x+d,\; y+b,\; z)\\ t_3(x,y,z) &= (x+e,\; y+f,\; z+c) \end{aligned}\) with \(a,b,c>0,\ d,e,f\in\mathbb{R}\) & \(\mathbb T^3\) \\
\hline
\(N_5\) & \(g(x,y,z) = (x+L,\; -y,\; z),\quad L>0\) & open M\"obius strip \(\times \mathbb{R}\) \\
\hline
\(N_6\) & \(\begin{aligned} t_1(x,y,z) &= (x,\; y+\alpha,\; z)\\ t_2(x,y,z) &= (x,\; y+\beta,\; z+\gamma)\\ g(x,y,z) &= (x,\; y+\delta,\; -z) \end{aligned}\) with \(\alpha,\gamma>0,\ \beta,\delta\in\mathbb{R}\) & \(\mathbb{R}\times\) Klein bottle \\
\hline
\(N_7\) & \(\begin{aligned} t_1(x,y,z) &= (x+\alpha,\; y,\; z)\\ t_2(x,y,z) &= (x,\; y+\beta,\; z)\\ g(x,y,z) &= (-x,\; -y,\; z+\tfrac12) \end{aligned}\) with \(\alpha,\beta>0\) & Half-turn space \\
\hline
\(N_8\) & \(\begin{aligned} t_1(x,y,z) &= (x+\alpha,\; y,\; z)\\ t_2(x,y,z) &= (x,\; y+\beta,\; z)\\ t_3(x,y,z) &= (x,\; y,\; z+\gamma)\\ g(x,y,z) &= (-x,\; y,\; z+\tfrac{\gamma}{2}) \end{aligned}\) with \(\alpha,\beta,\gamma>0\) & First non-orientable compact \(\mathbb{Z}_2\) \\
\hline
\(N_9\) & \(\begin{aligned} t_1(x,y,z) &= (x+\alpha,\; y,\; z)\\ t_2(x,y,z) &= (x,\; y+\beta,\; z)\\ t_3(x,y,z) &= (x,\; y,\; z+\gamma)\\ g(x,y,z) &= (x,\; -y,\; z+\tfrac{\gamma}{2}) \end{aligned}\) with \(\alpha,\beta,\gamma>0\) & Second non-orientable compact \(\mathbb{Z}_2\) \\
\hline
\(N_{10}\) & \(\begin{aligned} t_1(x,y,z) &= (x+\alpha,\; y,\; z)\\ t_2(x,y,z) &= \bigl(x-\tfrac{\alpha}{2},\; y+\tfrac{\sqrt{3}}{2}\alpha,\; z\bigr)\\ g(x,y,z) &= \bigl(x\cos120^\circ - y\sin120^\circ,\; x\sin120^\circ + y\cos120^\circ,\; z+\tfrac{\gamma}{3}\bigr) \end{aligned}\) with \(\alpha,\gamma>0\) & Third-turn space \\
\hline
\(N_{11}\) & \(\begin{aligned} t_1(x,y,z) &= (x+\alpha,\; y,\; z)\\ t_2(x,y,z) &= (x,\; y+\alpha,\; z)\\ g(x,y,z) &= (-y,\; x,\; z+\tfrac{\gamma}{4}) \end{aligned}\) with \(\alpha,\gamma>0\) & Quarter-turn space \\
\hline
\(N_{12}\) & \(\begin{aligned} t_1(x,y,z) &= (x+\alpha,\; y,\; z)\\ t_2(x,y,z) &= \bigl(x-\tfrac{\alpha}{2},\; y+\tfrac{\sqrt{3}}{2}\alpha,\; z\bigr)\\ g(x,y,z) &= \bigl(x\cos60^\circ - y\sin60^\circ,\; x\sin60^\circ + y\cos60^\circ,\; z+\tfrac{\gamma}{6}\bigr) \end{aligned}\) with \(\alpha,\gamma>0\) & Sixt-turn space \\
\hline
\(N_{13}\) & \(\begin{aligned} t_1(x,y,z) &= (x+2a,\; y,\; z)\\ t_2(x,y,z) &= (x,\; y+2b,\; z)\\ t_3(x,y,z) &= (x,\; y,\; z+2c)\\ g_1(x,y,z) &= (-x,\; y,\; z+c)\\ g_2(x,y,z) &= (x,\; -y,\; z+c) \end{aligned}\) with \(a,b,c>0\) & Hantzsche-Wendt manifold \\
\hline
\(N_{14}\) & Similar to \(N_{13}\) but with one generator orientation-reversing & First non-orientable compact \(\mathbb{Z}_2\times\mathbb{Z}_2\) \\
\hline
\(N_{15}\) & Another distinct family (different holonomy action) & Second non-orientable compact \(\mathbb{Z}_2\times\mathbb{Z}_2\) \\
\hline
\(N_{16}\) & \(\begin{aligned} g_1(x,y,z) &= (-x,\; y,\; z+\delta)\\ g_2(x,y,z) &= (x,\; -y,\; z+\delta) \end{aligned}\) with \(\delta>0\) & Non-compact \(\mathbb{Z}_2\times\mathbb{Z}_2\) (rank-0 translation) \\
\hline
\end{longtable}
\normalsize

\subsection*{Proof of Theorem \ref{T2}} 
First, note that the space of static potentials on a quotient manifold is precisely the subspace of \(span\{1,x,y,z\}\) that is invariant under the action of \(\Gamma\).

From the table, the flat \(3\)-manifolds for which the space of \(\Gamma\)-invariant functions in \(span\{1,x,y,z\}\) has dimension \(2\) are:
\begin{itemize}
    \item \(N_3: \mathbb{R} \times\mathbb T^2\): the group \(\Gamma\) consists of translations in \(y\) and \(z\); the invariant functions are \(c_0 + c_1 x\).
    \item \(N_5: \text{open M\"obius} \times \mathbb{R}\): the group \(\Gamma\) includes a glide reflection \((x+L, -y, z)\); the invariant functions are \(c_0 + c_3 z\).
    \item \(N_6: \mathbb{R} \times \text{Klein bottle}\): the group \(\Gamma\) includes translations in \(y\) and a glide reflection in \(y\) and \(z\); the invariant functions are \(c_0 + c_1 x\).
\end{itemize}
In each case the space of static potentials is generated by \(1\) and \(x\) (for \(N_3\) and \(N_6\)) or by \(1\) and \(z\) (for \(N_5\)). However, in all these cases the zero sets of linearly independent static potentials do not intersect.

In cases \(N_1\) and \(N_2\), \(\mathcal{P}\) is generated by \(\{1,x,y,z\}\) and \(\{1,y,z\}\), respectively. In both cases one can clearly find static potentials whose zero sets intersect (e.g., \(x\) and \(y\) for \(N_1\); \(y\) and \(z\) for \(N_2\)).

For all remaining flat \(3\)-manifolds, \(\mathcal{P}\) is one-dimensional (spanned by the constant function \(1\)), which contradicts the assumption that \(\mathcal{P}\) has dimension at least \(2\).

Therefore, under the stated assumptions, the only admissible dimensions for the space \(\mathcal{P}\) of static potentials are \(3\) and \(4\), corresponding to the manifolds \(N_1\) and \(N_2\).

\section{Appendix B} \label{appB}

Here, we give some details on the technical results from Section~\ref{compquest}. First, we describe the algorithm generating the table at the beginning of that section. Then, we provide the algorithm underlying our list of complete metrics of the form~\eqref{form}.

\subsection{Derivation of the table conditions} \label{table}

We study the function
\[
f(r)=k_1+\frac{k_2}{r}+k_3r^2,\qquad r>0,
\]
with real constants \(k_1,k_2,k_3\).  The positive roots of \(f\) are the positive roots of the cubic obtained by multiplying by \(r\):
\[
P(r)=k_3r^3+k_1r+k_2=0,\qquad r>0.
\]

\subsection*{1. Degenerate cases}
\begin{itemize}
  \item \(k_3=0\): \(P(r)=k_1r+k_2\).  A positive root exists iff \(k_1k_2<0\); then exactly one.
  \item \(k_2=0\): \(P(r)=r(k_3r^2+k_1)\).  For \(r>0\) we need \(k_3r^2+k_1=0\), which gives a positive root iff \(k_3k_1<0\); then exactly one.
  \item \(k_3\neq0,\;k_2\neq0\): proceed to the general analysis below.
\end{itemize}

\subsection*{2. General case \(k_3\neq0,\;k_2\neq0\)}
The cubic is depressed (no \(r^2\) term).  Its discriminant is
\[
\Delta = -4k_1^3k_3 - 27k_2^2k_3^2 = -k_3\,(4k_1^3+27k_3k_2^2).
\]
The number of real roots is determined by \(\Delta\). Namely,
\begin{itemize}
\item $\Delta>0\;\Rightarrow\; \text{three distinct real roots},$
\item $\Delta=0\;\Rightarrow\; \text{a multiple root},$
\item $\Delta<0\;\Rightarrow\; \text{one real root}.$
\end{itemize}

Since the sum of the three roots is \(0\), the signs of the roots are coupled.  By Vieta's theorem, the product of the roots equals \(-k_2/k_3\).  Descartes' rule of signs applied to \(P(r)\) and \(P(-r)\) gives bounds on the numbers of positive and negative roots.

A systematic case distinction based on the signs of \(k_3k_2\) and \(k_3k_1\) and the value of 
\(\Delta\) yields the table in the beginning of Section~\ref{compquest}.  
For instance:
\begin{itemize}
  \item When \(k_3k_2>0\) and \(k_3k_1<0\), the product of the three real roots (if all real) is 
    \(-\dfrac{k_2}{k_3}<0\). The cubic has three real roots iff \(\Delta>0\), i.e.
    \[
    -k_3(4k_1^3+27k_3k_2^2) > 0 \quad\Longleftrightarrow\quad k_3(4k_1^3+27k_3k_2^2) < 0.
    \]
    Under the sign conditions \(k_3k_2>0,\;k_3k_1<0\), this inequality is equivalent to
    \(27|k_3|k_2^2 < 4|k_1|^3\).  
    In this situation the three roots consist of two positive and one negative root, giving 
    \textbf{two distinct positive roots}.  
    If \(27|k_3|k_2^2 = 4|k_1|^3\) (i.e. \(\Delta=0\)), there is a double positive root and a simple negative root, i.e. 
    \textbf{one positive root (double)}.  
    If \(27|k_3|k_2^2 > 4|k_1|^3\) (i.e. \(\Delta<0\)), the cubic has only one real root (which is negative), so 
    \textbf{no positive root}.
  
  \item When \(k_3k_2<0\) and \(k_3k_1<0\), the product of the three real roots (if all real) is 
    \(-\dfrac{k_2}{k_3}>0\). Regardless of whether the cubic has one or three real roots, 
    there is always exactly \textbf{one positive root} (simple except when the discriminant vanishes, 
    in which case it is double).  
    When \(k_3k_2<0\) but \(k_3k_1>0\), one again finds \textbf{one positive root} (the cubic has 
    exactly one real root, which is positive).
  
  \item All remaining sign combinations (including \(k_3=0\) with \(k_1k_2\ge0\), \(k_2=0\) with 
    \(k_1k_3\ge0\), and \(k_3\neq0,k_2\neq0\) with \(k_3k_2>0,\;k_3k_1\ge0\), as well as 
    \(k_3k_2>0,\;k_3k_1<0\) with \(27|k_3|k_2^2 > 4|k_1|^3\)) yield \textbf{no positive root}.
\end{itemize}

The special subcases \(k_2=0\) or \(k_3=0\) are already covered by the degenerate analysis. Collecting all possibilities gives the conditions listed in the table in the beginning of Section~\ref{compquest}.

\subsection{Algorithm for determining completeness of the metric}\label{metriccomp}
Consider the metric
\[
g = f(r)\,dx^2 + \frac{dr^2}{f(r)} + r^2\,d\phi^2,
\qquad
f(r)=k_1+\frac{k_2}{r}+k_3r^2,
\]
with real constants \(k_1,k_2,k_3\).  The following algorithm identifies all parameter regimes for which the metric can be made geodesically complete by choosing appropriate periodic identifications.

\begin{enumerate}
    \item \textbf{Analyse \(f(r)\).} Use the results of the table in the beginning of Section~\ref{compquest}. See the explanation of them in the previous section.
    
    \item \textbf{Select admissible domains.}
    For each maximal interval \((a,b)\) with \(0\leqslant a<b\leqslant\infty\) where \(f>0\), consider the coordinate ranges:
    \begin{itemize}
        \item \(r\in (a,b)\),
        \item \(x\in\mathbb{R}\) or \(x\in \mathbb S^1\) (periodic),
        \item \(\phi\in \mathbb S^1\) (periodic, to be determined).
    \end{itemize}
    The metric is initially defined on the product of this interval with the circles (or line) for \(x\) and \(\phi\).

    \item \textbf{Check regularity at boundaries.}
    \begin{enumerate}
        \item \textbf{At \(r=0\):}  
        For \(r=0\) to be included, we need \(k_2=0\) and \(k_1>0\) so that \(f(0)=k_1>0\).  
        Then near \(r=0\),
        \[
        g \approx k_1 dx^2 + \frac{dr^2}{k_1} + r^2 d\phi^2,
        \]
        which is smooth if \(\phi\) has period \(2\pi/\sqrt{k_1}\).  The point \(r=0\) becomes a regular axis.

        \item \textbf{At a simple root \(r_0>0\) (\(f(r_0)=0,\;f'(r_0)\neq0\)):}  
        The coordinate singularity can be removed by taking \(x\) periodic with period \(4\pi/|f'(r_0)|\).  The surface \(r=r_0\) is then a smooth bolt (fixed point of \(\partial_x\)).  If the domain includes this root as an endpoint, the metric extends smoothly across it.

        \item \textbf{At a double root (\(f(r_0)=f'(r_0)=0\)):}  
        Such a root corresponds to a degenerate horizon and gives a cusp at infinite distance.  The metric extends to a complete manifold if one includes the boundary as a cusp; no period condition on \(x\) is required because the bolt degenerates.  This case is treated separately in the enumeration below.

        \item \textbf{As \(r\to\infty\):}  
        \begin{itemize}
            \item If \(k_3>0\): \(f(r)\sim k_3 r^2\); the radial distance diverges logarithmically, so the manifold is complete in this direction.
            \item If \(k_3=0\) and \(k_1>0\): \(f(r)\to k_1\); radial distance grows linearly, hence complete.
            \item If \(k_3<0\): \(f(r)\) becomes negative for large \(r\); the domain is bounded above by a root.
        \end{itemize}
    \end{enumerate}

    \item \textbf{Determine the periodicity of \(x\).}
    \begin{itemize}
        \item If no bolt is present, \(x\) can be either \(\mathbb{R}\) (non-compact) or \(\mathbb S^1\) with any period; completeness is unaffected.
        \item If a single bolt exists at \(r=r_0\), the period of \(x\) is forced to be \(4\pi/|f'(r_0)|\).
        \item If two bolts exist (e.g., at \(r=r_1\) and \(r=r_2\)), the periods must coincide: \(4\pi/|f'(r_1)| = 4\pi/|f'(r_2)|\).  This imposes an algebraic relation among the parameters, which for the cubic \(f(r)\) becomes \(r_1 r_2 = -k_1/(3k_3)\).
        \item If the only boundary is a double root (cusp), no period condition on \(x\) is required; \(x\) can be non-compact or periodic with any period.
    \end{itemize}

    \item \textbf{Enumerate all parameter regimes that yield a complete manifold.}
    Combine the root classification with the regularity conditions. This produces the families listed in the original classification:
    \begin{enumerate}
        \item \textbf{No roots:} \(k_2=0,\;k_1>0,\;k_3\geqslant 0\).  
        Domain: \(r\in[0,\infty)\).  \(\phi\) period \(2\pi/\sqrt{k_1}\), \(x\) arbitrary.  Geometry: \(\mathbb{E}^3\) (\(k_3=0\)) or \(\mathbb{H}^3\) (\(k_3>0\)).

        \item \textbf{One simple root, outer domain:}  \(f>0\) on \([r_0,\infty)\) with \(f'(r_0)>0\).  
        Occurs for several sign combinations (e.g., \(k_3>0,k_2<0\); \(k_3>0,k_2=0,k_1<0\); \(k_3>0,k_2>0\) with sufficiently negative \(k_1\); \(k_3=0,k_1>0,k_2<0\)).  
        Completeness requires \(x\) period \(4\pi/f'(r_0)\).

        \item \textbf{One simple root, inner domain:}  \(k_2=0,\;k_1>0,\;k_3<0\).  
        Domain: \(r\in[0,r_0]\) with \(r_0=\sqrt{-k_1/k_3}\).  \(\phi\) period \(2\pi/\sqrt{k_1}\), \(x\) period \(2\pi/\sqrt{-k_1k_3}\).  Compact manifold (squashed \(\mathbb S^3\)).

        \item \textbf{Two simple roots, bounded domain:}  \(k_3<0,\;k_1>0,\;k_2<0\) with \(27|k_3|k_2^2<4|k_1|^3\) (two distinct positive roots \(r_1<r_2\)).  
        The two bolts force the \(x\)-period to be \(4\pi/|f'(r_1)|\) and also \(4\pi/|f'(r_2)|\); for a smooth compact manifold these must be equal.  This equality is equivalent to \(r_1 r_2 = -k_1/(3k_3)\), which is a fine-tuning condition on the parameters.  When it holds, the domain \(r\in[r_1,r_2]\) yields a compact 3-manifold with two smooth bolts.

        \item \textbf{Double root (cusp):}  Occurs in two subcases:
        \begin{itemize}
            \item \(k_3>0,\;k_2=0,\;k_1=0\): double root at \(r=0\).  Domain: \(r\in(0,\infty)\) (since \(f(r)=k_3r^2>0\) for \(r>0\)).  The point \(r=0\) is a cusp at infinite distance; no period condition on \(x\) is needed.  \(\phi\) can have any period.
            \item \(k_3>0,\;k_2>0,\;k_1=-3\left(\dfrac{k_2^2k_3}{4}\right)^{1/3}\): double root at \(r_0=\left(\dfrac{k_2}{2k_3}\right)^{1/3}\).  Domain: \(r\in[r_0,\infty)\) with \(f>0\) for \(r>r_0\).  The cusp at \(r=r_0\) lies at infinite distance; again no period condition on \(x\) is required, and \(\phi\) can have any period.
        \end{itemize}
        In both cases the resulting manifold is complete and non-compact, with a cusp at the boundary.
    \end{enumerate}
    All other parameter choices either yield no region with \(f>0\) or lead to incomplete manifolds (e.g., a boundary that cannot be made regular).
\end{enumerate}

\bibliography{mybib}
\bibliographystyle{alpha}

\end{document}